\documentclass[10pt, a4paper]{article}
\usepackage[utf8]{inputenc}
\usepackage[left=2cm,right=2cm,top=1.5cm,bottom=1.5cm, a4paper]{geometry}
\usepackage{url}

\usepackage{amsmath,amssymb,fixmath,mathtools}
\usepackage{amsthm}            
\usepackage{graphicx}
\usepackage{bm}
\usepackage{comment}
\usepackage{soul}
\usepackage{xcolor}
\usepackage{tcolorbox}
\usepackage{parcolumns}
\usepackage{tabularx}
\usepackage[english]{babel}
\usepackage{lineno}  
\usepackage{tikz}
\usetikzlibrary{arrows.meta,positioning}

\def\E{\mathbb{E}}
\def\P{\mathbb{P}}
\def\R{\mathbb{R}}

\newcommand{\cA}{{\mathfrak A}}
\newcommand{\cS}{{\mathfrak S}}
\newcommand{\cX}{{\mathfrak X}}

\newcommand{\cG}{{\mathfrak G}}
\newcommand{\cM}{{\mathfrak M}}

\newcommand{\cF}{{\mathfrak F}}
\newcommand{\cK}{{\mathfrak K}}

\newcommand{\norm}[1]{\left\lVert#1\right\rVert}

\newtheorem{definition}{Definition}
\newtheorem{theorem}[definition]{Theorem}

\newtheorem{proposition}[definition]{Proposition}

\theoremstyle{definition}
\newtheorem{example}{Example} 

\theoremstyle{remark}
\newtheorem{remark}{Remark}

\newcommand{\cL}{{\cal L}}

\newcommand{\bgeqn}{\begin{eqnarray}}
\newcommand{\edeqn}{\end{eqnarray}}
\newcommand{\bgeq}{\begin{eqnarray*}}
\newcommand{\edeq}{\end{eqnarray*}}
\newcommand{\bec}{\begin{center}}
\newcommand{\enc}{\end{center}}

\newcommand{\var}{{\rm VaR}} 

\newcommand{\cE}{{\cal E}}

\newcommand{\half}{ \mbox{\small$\frac{1}{2}$}}

\def\dom{{\rm dom}}

\def\cl{{\rm cl}}

\def\ess {{\rm ess\, sup}}
\def\esi {{\rm ess\, inf}}

\def\cvar{{\sf CVaR}}

\newcommand{\bbr}{{\Bbb{R}}}
\newcommand{\bbp}{{\Bbb{P}}}
\newcommand{\bbe}{{\Bbb{E}}}

\makeatletter
\newsavebox{\@brx}
\newcommand{\llangle}[1][]{\savebox{\@brx}{\(\m@th{#1\langle}\)}%
  \mathopen{\copy\@brx\kern-0.5\wd\@brx\usebox{\@brx}}}
\newcommand{\rrangle}[1][]{\savebox{\@brx}{\(\m@th{#1\rangle}\)}%
  \mathclose{\copy\@brx\kern-0.5\wd\@brx\usebox{\@brx}}}
\makeatother

\def\plus{{\scriptscriptstyle +}} \def\minus{{\scriptscriptstyle -}}
\def\text#1{\;\,\hbox{#1}\;\,}    
\def\lset{\big\{\,}    \def\mset{\,\big|\,}   \def\rset{\,\big\}}
\def\Lset{\Big\{\,}    \def\Mset{\,\Big|\,}   \def\Rset{\,\Big\}}
\outer\def\proclaim #1. #2
   \par{\medbreak\noindent{\bf#1.\enspace}{\sl#2}\par
  \ifdim\lastskip<\medskipamount \removelastskip\penalty55\medskip\fi}

\def\paritem#1{\vskip0cm\noindent\hskip12pt{{\rm #1}}\hskip5pt}
\def\eop{\hfill{$\vcenter{\hrule height1pt \hbox{\vrule width1pt height5pt
   \kern5pt \vrule width1pt} \hrule height1pt}$} \medskip}
\def\low#1{{\lower1pt \hbox{$\scriptstyle #1$}}}
\def\high#1{{\raise1pt \hbox{$\scriptstyle #1$}}}
 
\def\implies{\quad\hbox{$\Longrightarrow$}\quad} 
\def\iff{\quad\hbox{$\Longleftrightarrow$}\quad} 

\def\dom{\mathop{\rm dom}\nolimits} 
\def\argmin{\mathop{\rm argmin}}   
\def\sumn{\sum\nolimits}      
 
\def\half{{{}\raise 1pt \hbox{$\frac{\scriptstyle 1}{\scriptstyle 2}$}}}
\def\eqalign#1{\begin{array}{lcr} #1 \end{array}}
\def\reals{{I\kern-.35em R}} \def\mdot{{\kern-.02em\cdot\kern-.04em}}

 \def\newpage{\vfill\eject}

\def\cA{{\cal A}}  \def\cC{{\cal C}} \def\cD{{\cal D}} 
\def\cE{{\cal E}} \def\cF{{\cal F}} \def\cG{{\cal G}}  
\def\cJ{{\cal J}} \def\cK{{\cal K}} \def\cL{{\cal L}} \def\cM{{\cal M}} 
\def\cP{{\cal P}} \def\cQ{{\cal Q}} \def\cR{{\cal R}} \def\cS{{\cal S}} 
\def\cV{{\cal V}}  \def\cX{{\cal X}} 

\def\cvar{{\rm CVaR}}

\def\<x>{\langle\!\langle\mathbf{x}\rangle\!\rangle}
\def\l<{\langle\!\langle}
\def\r>{\rangle\!\rangle}
\usepackage{hyperref} 
\hypersetup{
    colorlinks=true,
    linkcolor=blue,
    filecolor=magenta,      
    urlcolor=cyan,
    pdftitle={Overleaf Example},
    pdfpagemode=FullScreen,
    }

\title{The Risk Quadrangle in Optimization: An Overview with Recent Results and Extensions}
\author{Bogdan Grechuk\thanks{School of Computing and Mathematic Sciences, University of Leicester, Leicester LE1 7RH, UK;\\ email: \url{bg83@leicester.ac.uk}} \and Anton Malandii\thanks{Department of Applied Mathematics and Statistics, State University of New York, Stony Brook, NY 11794, USA;\\ emails: \url{anton.malandii@stonybrook.edu}, \url{stanislav.uryasev@stonybrook.edu}} \and Terry Rockafellar\thanks{University of Washington, Department of Applied Mathematics, 408 L Guggenheim Hall, PS Box 352420, Seattle, WA 98195, USA;\\ email: \url{rtr@math.washington.edu}}
\and
Stan Uryasev\footnotemark[2]
}

\begin{document}
\maketitle
\begin{abstract}
This paper revisits and extends the 2013 development by Rockafellar and Uryasev of the Risk Quadrangle (RQ) as a unified scheme for integrating risk management,
optimization, and statistical estimation. The RQ features four stochastics-oriented functionals --- risk, deviation, regret, and error, along with an 
associated statistic, and articulates their revealing and in some ways
surprising interrelationships and dualizations. Additions to the RQ framework that have come to light since 2013 are reviewed in a synthesis focused
on both theoretical advancements and practical applications. New quadrangles
--- superquantile, superquantile norm, expectile, biased mean, quantile 
symmetric average union, and $\varphi$-divergence-based quadrangles --- offer 
novel approaches to risk-sensitive decision-making across various fields such as machine learning, statistics, finance, and PDE-constrained optimization. The theoretical contribution comes in axioms for ``subregularity'' relaxing ``regularity'' of the quadrangle functionals, which is too restrictive for some applications. The main 
RQ theorems and connections are revisited and rigorously extended to this more ample framework. Examples are provided in portfolio optimization, 
regression, and classification, demonstrating the advantages and the 
role played by duality, especially in ties to robust optimization and
generalized stochastic divergences.
\end{abstract}


\section{Introduction}
This paper extends the work of Rockafellar and Uryasev \cite{MR3103448} and reviews the new results published since 2013. Although the material in the paper is self-contained and can be read independently, readers are strongly encouraged to familiarize themselves with the foundational content of \cite{MR3103448} for a still broader perspective. 

\paragraph{Background.} The Risk Quadrangle (RQ) is a unified framework that integrates risk management, optimization, and statistical estimation by grouping four axiomatically defined stochastics-oriented functionals: risk, deviation, regret, and error connected by a so-called statistic.  For a random variable 
$X$, \emph{risk} provides a numerical surrogate for the overall hazard in 
$X$, \emph{deviation} measures the ''nonconstancy`` in $X$, \emph{regret} 
assesses the distaste in facing the mix of good-and-bad outcomes of $X$, 
and \emph{error} quantifies the ''nonzeroness`` in $X$, not necessarily with symmetry between positive and negative outcomes.  The \emph{statistic} 
identifies the constant closest to $X$ with respect to a particular form
of error functional, but also gives the level of trade-off between future loss and immediate acceptance that is appropriate for a particular form of
regret functional.  The fruitful interplay between these functionals supports a vast 
spectrum of applications across all engineering areas where probabilistically modeled uncertainty is involved.  Among them, to name a few, are quantitative finance, reliability engineering, mechanical engineering, medical imaging, and machine learning.

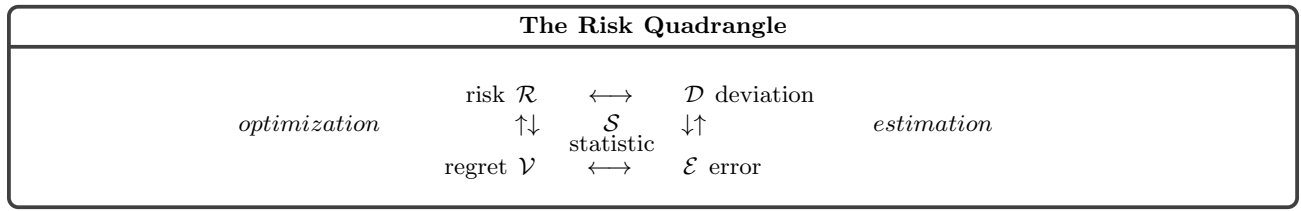
\begin{tcolorbox}[title = \small The Risk Quadrangle, halign title = center,colback=white,  colbacktitle=white, coltitle=black, fonttitle=\bfseries]
\vspace{-0.2cm}
\small
     \begin{center}
$$\eqalign{
\hskip91pt \text{risk} \cR \hskip15pt\,\longleftrightarrow \,\hskip15pt\cD \text{deviation} &\cr
\hskip10pt optimization \hskip48pt 
\uparrow\downarrow\hskip05pt\underset{\text{statistic}}{\cS}\hskip04pt\downarrow\uparrow 
    \hskip60pt estimation \hskip28pt &\cr
\hskip83pt \text{regret} \cV\hskip15pt\,\longleftrightarrow \,\hskip15pt\cE \text{error} &\cr
}$$
\end{center}
\end{tcolorbox}
\begin{tcolorbox}[title = \small General Relationships, halign title = center,colback=white,  colbacktitle=white, coltitle=black, fonttitle=\bfseries]

\small
   $$\cR(X) = \cD(X) + \bbe[X] \qquad \cV(X) = \cE(X) + \bbe[X]$$
   $$\hspace{0.9cm}\cD(X)= \min\limits_C\!\lset\cE(X-C)\rset \qquad
    \cR(X)= \min\limits_C\!\lset C+\cV(X-C)\rset$$
      $$\hspace{1.1cm}\argmin_C\!\lset\cE(X-C)\rset=
           \cS(X)=\argmin_C\!\lset C+\cV(X-C)\rset$$
\end{tcolorbox}
The relationship between risk and regret allows: 1) constructing new risk measures and quadrangles; 2) building efficient algorithms for risk optimization (see e.g., \cite{CVaR, CVaR2, ben-tal_old-new_2007, Expect, Anton2024expectile}, and more). The relation between deviation and error links statistical estimation and risk management, resulting in concepts such as risk-tuned regression and risk tracking (see \cite{MR2442649, MR3357644}). This relation shows equivalences among various types of regressions and provides robust and efficient estimation techniques (see \cite{CVaRRegr, CVaRRegGolod, Anton2024expectile, Anton2024biased}). For more information on risk-adaptive approaches, see \cite{MR3103448, Royset2022}.

\paragraph{Motivation.} The motivation of this paper is twofold: on the \emph{practical side}, it aims to offer engineers and practitioners robust, implementable methods and models for regression and risk estimation under real-world constraints; on the \emph{theoretical side}, it seeks to advance foundational principles in stochastic optimization, distributionally robust optimization, convex analysis, and statistical estimation, targeting pure mathematicians who explore deep relationships within these fields.

We highlight here that \cite{Royset2022} provides an excellent, wide-ranging survey of risk-adaptive approaches to stochastic optimization. However, our contribution is complementary and distinct: this article is \emph{not} another overview of quadrangle ideas but a \emph{rigorous mathematical development} of the framework. We focus on the RQ as the organizing framework and deliver (i) \emph{axiomatic extensions} that sharpen and generalize the framework, (ii) new \emph{duality results} formulated via \emph{generalized stochastic divergences} tailored to quadrangle elements, and (iii) a set of \emph{recent applications} from the literature that operationalize the framework in modern modeling and algorithms. This focus consolidates and advances the theory while clarifying modeling templates and algorithmic consequences.

\medskip

\textbf{\textit{Practical aspects.}}  The original work in 
\cite{MR3103448} established the foundation of the RQ, formulating key 
theorems and relationships that unify its elements within a single 
mathematical structure.  The innovations were largely theoretical, without 
extensive demonstrations of practical applications or empirical benefits, but in the last decade, a series of other works has begun to fill the gap.   The RQ scheme has been leveraged to answer application-driven 
questions with practical results in risk management, statistical estimation, machine learning, fairness-aware machine learning, and beyond.  This has encompassed new additions such as the superquantile quadrangle \cite{CVaRRegr}, 
the superquantile norm quadrangle \cite{CVaRNorm2}, 
the expectile quadrangles \cite{Anton2024expectile}, 
the biased mean quadrangle \cite{Anton2024biased}, and the quantile symmetric average union quadrangle \cite{Anton2022SVR}.

In particular, these recent contributions illustrate how the RQ provides a powerful framework for developing new regression estimators. For example, the expectile and least squares regressions (known to be quadratic programs) have been reduced to linear programs, thereby enabling fast algorithms that also accommodate mixed-integer constraints --- an essential feature in many real-world applications. The biased mean quadrangle introduced a novel estimator for the so-called biased mean, broadening the scope of classical methods in robust statistics. Likewise, the superquantile norm quadrangle has uncovered fundamental connections between support vector regression (SVR) and distributionally robust optimization, while the quantile symmetric average union quadrangle has shown that SVR estimates the average of two symmetric quantiles, thus clarifying its theoretical and practical strengths and limitations.

Another active area of research, fairness-aware machine learning, has also significantly benefited from the development of the axiomatic theory of risk and deviation measures. In particular, \cite{williamson2019fairness} defined fairness as a deviation in subgroup risks, which enabled the reduction of fairness-aware classification to a convex optimization problem. This problem can be efficiently solved by leveraging the relationship between risk and regret.

Finally, epi-regularization of risk measures \cite{kouri2020epi} has emerged as a unifying thread in stochastic optimization for turning nonsmooth risk functionals into smooth and algorithm-friendly surrogates while preserving their decision‐relevant structure. By smoothing coherent and/or regular risk measures (e.g., conditional value-at-risk (CVaR), expectile) via infimal convolution, one obtains differentiable approximations that epi-converge to the original problems, retain level-set geometry, and enable first- and second-order methods, proximal/Bregman schemes, and stochastic algorithms with clear convergence guarantees \cite{cvaropt}. On the dual side, epi-regularization corresponds to controlled relaxation of risk envelopes, yielding stable sample-average estimators. These developments matter in engineering --- where reliability, safety, and performance are co-optimized in power systems, structural and aerospace design, model-predictive control, and autonomous systems --- and in mathematics, where they connect variational convergence, monotone operator theory, and generalized stochastic divergences, motivating new duality results and principled algorithm design.

Despite this progress, the emerging body of work remains scattered across multiple publications, each focusing on a specific quadrangle. Researchers and practitioners seeking to exploit these new developments lack a single, cohesive resource detailing how the RQ can be specialized, extended, and integrated into broad classes of optimization and estimation problems. 

This paper synthesizes recent advancements in the RQ framework, highlighting and extending the developments and results mentioned above, as well as the benefits they offer for risk-sensitive decision-making. The following Map of Applications illustrates several representative examples from machine learning, statistics, quantitative finance, and risk-averse PDE-constrained optimization (see, for example, \cite{Antil2018}), showcasing the breadth of the framework. The following Portfolio Optimization example for risk and deviation is based on general relationships between quadrangle elements, as described in Theorem \ref{Quadrangle Theorem}. The PDE-constrained Expectile Minimization example is based on the Expectile Quadrangle considered in Example \ref{ExpectileExample}. The Linear Regression example is based on the Regression Theorem \ref{Regression Theorem}. The Support Vector Classification example originates from the work \cite{NuSVMasCVaR}.
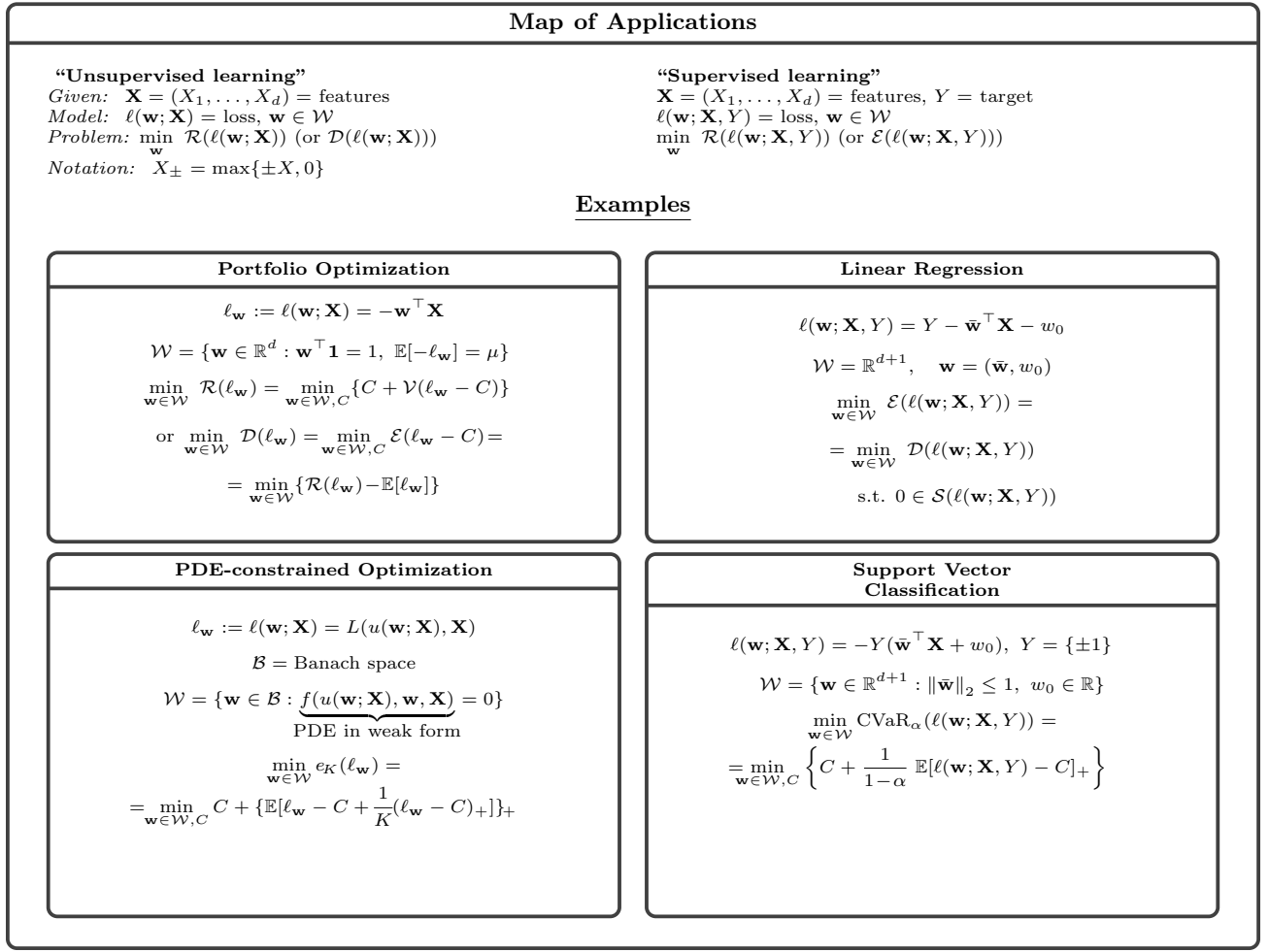
\begin{tcolorbox}[title = \small Map of Applications, halign title = center,colback=white,  colbacktitle=white, coltitle=black, fonttitle=\bfseries]
\begin{parcolumns}[colwidths={ 1=0.48\textwidth, 2=0.48\textwidth}]{2}
\scriptsize
        \colchunk{
\!\!        \textbf{``Unsupervised learning''}
      
    
    \emph{Given:} \  $\mathbf{X} = (X_1,\ldots, X_d)=$ features
    

    \emph{Model:} \  $\ell(\mathbf{w};\mathbf{X})=$ loss, $\mathbf{w} \in \mathcal{W}$


    \emph{Problem:} $\min\limits_{\mathbf{w}} \ \cR(\ell(\mathbf{w};\mathbf{X}))$ (or $ \cD(\ell(\mathbf{w};\mathbf{X}))$)

\emph{Notation:} \ $X_\pm=\max\{\pm X,0\}$

    }
     \colchunk{
     \textbf{\!\!``Supervised learning''}
  

    
    $\mathbf{X}= (X_1,\ldots, X_d)=$ features, $Y=$ target


    $\ell(\mathbf{w}; \mathbf{X},Y)=$ loss, $\mathbf{w} \in \mathcal{W}$


    $\min\limits_{\mathbf{w}} \ \cR(\ell(\mathbf{w};\mathbf{X},Y))$ (or $ \cE(\ell(\mathbf{w};\mathbf{X},Y))$)
    }
    \end{parcolumns}
\vspace{0.1 cm}
    \begin{center}
        \underline{\small \textbf{Examples}}
    \end{center}
    \begin{parcolumns}[colwidths={ 1=0.49\textwidth, 2=0.49\textwidth}]{2}
    \vspace{-0.3cm}
\scriptsize
     \colchunk{
\begin{tcolorbox}[title = Portfolio Optimization, halign title = center,colback=white,  colbacktitle=white, coltitle=black, fonttitle=\bfseries, height=4cm]
 \vspace{-0.2 cm}
$$\ell_\mathbf{w}:=\ell(\mathbf{w};\mathbf{X}) = -\mathbf{w}^\top\mathbf{X}$$
$$\hspace{-0.1cm}\mathcal{W} = \{ \mathbf{w} \in \bbr^d: \mathbf{w}^\top\textbf{1}  = 1, \ \bbe[-\ell_\mathbf{w}] = \mu\}$$
$$\!\!\!\!\min_{\mathbf{w} \in \mathcal{W}} \ \cR(\ell_\mathbf{w}) = \min_{\mathbf{w} \in \mathcal{W}, C} \{ C + \cV(\ell_\mathbf{w}-C)\}
$$
$$
\!\!\!\!\!\text{or} \!\min_{\mathbf{w} \in \mathcal{W}} \ \cD(\ell_\mathbf{w}) =\!\min_{\mathbf{w} \in \mathcal{W}, C} \cE(\ell_\mathbf{w}-C)\! =
$$
$$
= \min_{\mathbf{w} \in \mathcal{W}} \{ \cR(\ell_\mathbf{w})\!-\!\bbe[\ell_\mathbf{w}] \}
$$

\end{tcolorbox}

\begin{tcolorbox}[title =   PDE-constrained Optimization, halign title = center,colback=white,  colbacktitle=white, coltitle=black, fonttitle=\bfseries, height=5cm]
 $$\ell_\mathbf{w}:=\ell(\mathbf{w};\mathbf{X}) = L(u(\mathbf{w}; \mathbf{X}), \mathbf{X})$$ $$\mathcal{B} = \textrm{Banach space} $$
 $$\mathcal{W} = \{\mathbf{w} \in \mathcal{B}: \!\!\!\! \underset{\text{PDE in weak form}}{\underbrace{f(u(\mathbf{w}; \mathbf{X}), \mathbf{w}, \mathbf{X})}}
 \!\!\!\! = 0 \}$$
 $$\min_{\mathbf{w} \in \mathcal{W}} e_{\!K}(\ell_\mathbf{w})=
 $$
 \vspace{-0.3 cm}
 $$\!\!\!\!\!\!=\!\!\min_{\mathbf{w} \in \mathcal{W},C} 
 C+\{\bbe[\ell_\mathbf{w}-C + \!\frac{1}{K}\!(\ell_\mathbf{w}-C)_+]\}_{\!+}
 $$

\end{tcolorbox}
}

     \colchunk{
\begin{tcolorbox}[title = Linear Regression, halign title = center,colback=white,  colbacktitle=white, coltitle=black, fonttitle=\bfseries, height=4cm]
$$\ell(\mathbf{w}; \mathbf{X},Y) = Y -  \bar{\mathbf{w}}^\top\mathbf{X} - w_0$$
$$ \mathcal{W} = \bbr^{d+1}, \quad \mathbf{w} = (\bar{\mathbf{w}}, w_0)$$
$$ \min_{\mathbf{w} \in \mathcal{W}} \ \cE(\ell(\mathbf{w}; \mathbf{X},Y))=$$
$$=\min_{\mathbf{w} \in \mathcal{W}} \ \cD(\ell(\mathbf{w}; \mathbf{X},Y))$$
$$ \qquad \text{s.t.} 0 \in \cS(\ell(\mathbf{w}; \mathbf{X},Y))
$$
\end{tcolorbox}

\begin{tcolorbox}[title = Support Vector\\ Classification, halign title = center,colback=white,  colbacktitle=white, coltitle=black, fonttitle=\bfseries, height=5cm]

$$\hspace{-0.3cm}\ell(\mathbf{w}; \mathbf{X},Y) = -Y(\bar{\mathbf{w}}^\top\mathbf{X}+ w_0), \ Y = \{\pm 1\} $$
$$ \mathcal{W} = \{\mathbf{w} \in \bbr^{d+1}: \norm{\bar{\mathbf{w}}}_2 \leq 1, \ w_0 \in \bbr \}$$
$$\min_{\mathbf{w} \in \mathcal{W}} \cvar_\alpha(\ell(\mathbf{w}; \mathbf{X},Y)) =
$$
\vspace{-0.25 cm}
$$\hspace{-0.4cm} = \hspace{-0.2cm}\min_{\mathbf{w} \in \mathcal{W}, C} \left\{ C + \frac{1}{1\!-\!\alpha}\; \bbe[\ell(\mathbf{w};\mathbf{X},Y)-C]_+\right\}
$$
\end{tcolorbox}

     }
    \end{parcolumns}
    
    \end{tcolorbox}

\medskip
\textbf{\textit{Theoretical aspects.}}  Practical challenges and advances 
have in turn stimulated the theoretical development of the framework ---
especially its axiomatic foundations.  The basic theorems and relationships 
have accordingly needed updating beyond their original forms in 
\cite{MR3103448}.  Interest in generalized stochastic divergences in 
connection with distributional robustness, for example, has led to the 
discovery that such divergence functionals are actually dual to risk functionals, but of a sort not covered by the original scheme.  That has pointed to the need for some relaxation in the original axioms.

The stochastics-oriented functionals that comprise the quadrangle are axiomatically defined (see Axioms for Deployment in Various Combinations), with their axioms deeply rooted in convex 
analysis and classical statistics.  The axiomatic nature of the framework elevates the optimization problem settings beyond the expectation-type objectives that have long been common in statistics and, more recently, in machine learning.  The significance of objectives other than expected loss in machine learning became evident with the development of $\nu$--Support Vector Machines in \cite{NewSVM}.

The rapid development and global influence of machine learning have revealed the need for axiomatic revisions in the RQ setting.  Specifically, several 
useful error quantifiers such as the Vapnik error (or expected 
$\varepsilon$--insensitive loss), $\cE(X) = \bbe[|X|-\varepsilon]_+, 
\ \varepsilon \geq 0$ and the superexpectation error, 
$\cE(X) = \big\{\bbe[X_-] - x_+, \bbe[X_+] - x_-\big\}, \ x \in \bbr$, do not 
satisfy the strict positivity axiom (A9) (see Axioms for Deployment in Various Combinations), which 
stimulates an appropriate relaxation of this particular axiom, followed by 
axioms (A6)--(A8). The strict positivity axiom states that an error of a 
random variable $X$ should always be positive whenever $X$ is nonzero, 
which holds for most classical errors, such as mean squared error, mean 
absolute error, etc.  However, in practice, there may exist so-called 
``error-insensitive'' regions, where errors below a certain threshold are 
ignored.  The most famous example is in $\varepsilon$--Support Vector 
Regression (cf. \cite{SVR, SVR2}), where the Vapnik error is minimized. 
To address the aforementioned issues, this paper introduces 
\emph{subregularity axioms}, where the concept of subaversity is used to relax the (A6)--(A9).   The proofs of the fundamental theorems have to be adjusted to incorporate this extension.

\begin{tcolorbox}[title = \small Axioms for Deployment in Various Combinations, halign title = center,colback=white,  colbacktitle=white, coltitle=black, fonttitle=\bfseries]
For a functional $\cF: \cL^p \to (-\infty, \infty]$, a subset of the following properties may be of interest (cf. Subsection~\ref{subsec: axioms and defs} for formal definitions):
\begin{itemize}
    \item[(A1)] \emph{Convexity:} \quad 
   ${\cal F}(\lambda X + (1-\lambda) Y) \leq 
   \lambda {\cal F}(X) + (1-\lambda){\cal F}(Y)$ for all 
     $X,Y \in {\cal L}^p$ and every $\lambda \in [0,1]$.

    \item[(A2)] \emph{Lower semicontinuity:} \quad 
    $\{\cF \leq C\}:=\{X \in \cL^p: \cF(X) \leq C\}$ is closed for all 
                $C \in \bbr.$ 
    \item[(A3)] \emph{Constant fidelity:} \quad 
          $\cF(X) = C \in \bbr$ if $\bbp(X = C) = 1.$

    \item[(A4)] \emph{Constant neutrality:} \quad 
                $\cF(X) = 0 \in \bbr$ if $\bbp(X = C) = 1.$

    \item[(A5)] \emph{Zero fidelity:} \quad 
              $\cF(X) = 0 \in \bbr$ if $\bbp(X = 0) = 1.$

    \item[(A6)] \emph{Aversity:} \quad 
            $\cF(X) > \bbe[X]$ for all $X \neq C.$
     
    \item[(A7)] \emph{Zero aversity:} \quad 
                 $\cF(X) > \bbe[X]$ for all $X \neq 0.$

    \item[(A8)] \emph{Nonconstant positivity:} \quad 
         $\cF(X) > 0$ for all $X \neq C.$
    
    \item[(A9)] \emph{Strict positivity:} \quad $\cF(X) > 0$ for all $X \neq 0.$
    
    \item[(A10)] \emph{Monotonicity:} \quad 
            $\cF(X) \leq \cF(Y)$ when ${\mathbb P}(X \leq Y)=1$.

    \item[(A11)] \emph{Positive homogeneity:} \quad 
           $\cF(\lambda X)=\lambda\cF(X)$ when $\lambda>0$.

\end{itemize}

\end{tcolorbox}
\begin{tcolorbox}[title = \small Regular Quadrangle Axioms, halign title = center,colback=white,  colbacktitle=white, coltitle=black, fonttitle=\bfseries]
\begin{itemize}
    \item[($\cR$)] \emph{Regular risk}: (A1), (A2), (A3), (A6).
    \item[($\cV$)] \emph{Regular regret}: (A1), (A2), (A5), (A7).
    \item[($\cD$)] \emph{Regular deviation}: (A1), (A2), (A4), (A8).
     \item[($\cE$)] \emph{Regular error}: (A1), (A2), (A5), (A9).
\end{itemize}
\end{tcolorbox}
\medskip

Originally in \cite{MR3103448}, the definitions of regular error and regret 
also included limit conditions.  For $\cE$, this required that 
$\lim_{k\to \infty}{\mathbb E}[X_k]=0$ for any sequence of random variables (r.v.s) 
$\{X_k\}_{k=1}^\infty$ satisfying $\lim_{k\to \infty}\cE[X_k]=0$, and there
was an analogous condition for $\cV$.  Those conditions were needed back then 
for the proofs of particular theorems, but in the meantime they have been found 
to be superfluous --- other proofs succeed without invoking them.  
Presenting the RQ framework in that simpler form is one of the contributions of this paper.

For the orientation of readers who may be familiar with risk measures $\cR$ 
that are \emph{coherent}, but perhaps new to risk measures that are 
\emph{regular}, an explanation of the relationship may be helpful.  As 
introduced in the pioneering days of risk theory, a coherent measure of 
risk had to satisfy axioms that, although stated differently, were
equivalent to the combination of (A1), (A2), (A3), (A10), and (A11) for
$\cR$ in place of $\cF$ (although (A2) did not appear from the start, 
because the original context was such that finiteness and convexity of the functionals 
entailed their continuity). Later, there was an incentive on many fronts to 
drop (A11) from this list, and then $\cR$ was called, by some, simply
a convex measure of risk.  The trouble is that there are very prominent
instances of risk measure $\cR$, such as the one in the Standard Mean-Based Quadrangle (see Example \ref{ex:standard}), that are convex but lack the monotonicity
in (A10).   That monotonicity, however, is a watershed requirement for
dualization that reflects in terms of alternative probability distributions 
compared to the nominal one.  Rockafellar suggested it would make better
sense to tie coherency to that monotonicity and adjust the terminology by
speaking of risk measures satisfying (A1), (A2), (A3), and (A10) as
coherent in the general sense.   But if monotonicity is so important, why
is (A10) absent from the definition of a regular risk measure $\cR$ 
in \cite{MR3103448}?  That is because the Standard Quadrangle, with its
long-running usage and deep connection to classical statistics, had to be
admitted to the RQ picture despite its lack of (A10), in particular for 
bringing out that deficiency.
\medskip

\textbf{\textit{Incentives coming from stochastic divergences.}}  
The stochastic modeling of uncertainty entails an underlying probability 
space $(\Omega, \cM, \bbp_0),$ which is used to define the random losses 
(or costs), $\ell(\mathbf{w}, \omega), \ \omega \in \Omega,$ as real-valued 
random variables.  Specifically, the probability of a random loss being 
less than a certain numerical threshold is measured by the cumulative 
distribution function defined by the underlying probability measure, i.e., 
$F(x) = \bbp_0(\{\omega \in \Omega: \ell(\mathbf{w},\omega) \leq x\}).$ 
When the goal is to find a decision $\mathbf{w}$ that ``minimizes'' the 
loss function, the stochastic nature of this function raises the question: 
``minimizes in which sense?''  The answer to this question significantly 
relies on the information regarding the underlying probability $\bbp_0,$ 
which is usually either unknown or only partially available.

When the natural designation of $\bbp_0$ is at hand, the classical approach 
suggests minimizing the average loss, i.e., 
$$\min_{\mathbf{w}} \ \bbe_{\bbp_0}[\ell(\mathbf{w},\omega)].$$ 
The average loss minimization is often referred to as \emph{risk-neutral} 
optimization.  On the other hand, if there is a lack of information 
regarding the underlying probability, the \emph{robust} approach suggests 
minimizing the loss associated with the worst outcome $\omega \in \Omega,$ 
i.e.,
 $$ \min_{\mathbf{w}}\max_{\omega \in \Omega} \ \ell(\mathbf{w},\omega).$$ 
While this approach has its merits and successes, it can also be overly 
conservative and, consequently, costly.

The \emph{distributionally robust} approach offers a sort of compromise 
between neutrality and conservativeness.  Rather than relying solely on a 
single distribution $\bbp_0$ or avoiding probabilities entirely, we can 
consider sets of alternative distributions $\bbp \in \mathfrak{P}$ and 
minimize the expected value with respect to the worst one, i.e., 
 $$\min_{\mathbf{w}}\max_{\bbp \in \mathfrak{P}} 
            \ \bbe_{\bbp}[\ell(\mathbf{w},\omega)].$$
Consequently, it raises the question of how to construct the set of 
probability alternatives $\mathfrak{P}$.   It turns out that the answer is 
intimately connected to the duality theory for coherent measures of risk 
$\cR$ in the basic sense, satisfying (A1), (A2), (A3), (A10) and (A11),
and the concept of distributional robustness in that way is an integral part 
of the RQ picture.  A key there to constructing sets $\mathfrak{P}$ consisting of probability alternatives is to take those sets to be ``neighborhoods''
of nominal probability measure $\bbp_0$ with respect to some sort of
``distance'' concept.  Such distances are furnished by general \emph{stochastic divergence} functionals.  That much has been understood by practitioners in machine learning and elsewhere, but here there will be something more.   Stochastic divergence functionals will be given an axiomatic definition and seen as duals to special risk measures that are not quite regular, but fit with a relaxation of regularity in addition to being coherent in the general sense.  
  
\paragraph{Outline.} The rest of the paper is organized as follows. 
Section \ref{sec:examples} provides a list of examples of risk quadrangles 
highlighting the framework's scope and motivating its further extensions. Section \ref{sec:extended} reviews and extends the main properties and 
relationships of the subregular RQ. Specifically, subregularity axioms and 
definitions are given in Subsection \ref{subsec: axioms and defs}. 
The discussion on the fundamental theorems --- cornerstones of the framework --- is a subject of 
Subsection \ref{subsec:deviation}.  The primal aspects of quadrangle construction are covered in Subsection \ref{sub: primal representation}. Dual representation and conjugate duality 
of the subregular risk, deviation, regret, and error are discussed in 
Subsection \ref{sec:dual}, and the concept of parent functionals and corresponding positive homogeneous families is presented in Subsection \ref{sec:gen families}. Section \ref{sec: Regression} is devoted to discussing generalized regression and statistical estimation in the subregular RQ framework. Finally, Section \ref{sec: SO and DR} provides an overview of the concepts of robust and distributionally robust optimization together with new results on epi-regularization within the subregular RQ framework.

\section{Examples of Risk Quadrangles}\label{sec:examples}

Following \cite{MR3103448}, we start with the examples highlighting the framework's scope and motivating its further extension. We begin with a few main examples from \cite{MR3103448}, and then continue with new examples.

\begin{example}[Standard Mean-Based Quadrangle, $\lambda > 0$]\label{ex:standard} 
This is Example 1 in \cite{MR3103448}, where it is called the Mean-based 
Quadrangle. 

\begin{tcolorbox}[title=1. Standard Mean-Based Quadrangle, halign title=center, colback=white,  colbacktitle=white, coltitle=black, fonttitle=\bfseries, height=4cm]
\vspace{-0.5cm}
	\begin{align*}
		\cS(X) & = \bbe[X] = \!\!\text{mean} \\
		\cR(X) & = \bbe[X] + \lambda \sigma(X) = \!\!\text{safety margin tail risk}\\
		\cD(X) & = \lambda \sigma(X) = \!\!\text{standard deviation, scaled}\\
		\cV(X) &= \bbe[X] + \lambda \|X\|_2 = \!\!\text{$\cL^2$-regret, scaled}\\
		\cE(X) &= \lambda \|X\|_2 = \!\!\text{$\cL^2$-error, scaled}
	\end{align*}
\end{tcolorbox}
\end{example}


\begin{example}[Quantile-Based Quadrangle, $\alpha \in (0,1)$] 
This is Example 2 in \cite{MR3103448}. Recall that for any r.v. $X$ its 
cumulative distribution function is $F_X(x)=\bbp (X\leq x)$, and its 
$\alpha$-quantile is $q_\alpha(X) = [q^-_\alpha(X), q^+_\alpha(X)]$, where 
$q^-_\alpha(X)=\sup\{x\,|\,F_X(x)<\alpha\}$ and 
$q^+_\alpha(X)=\inf\{x\,|\,F_X(x)>\alpha\}$. 
Recall also notation $X_+=\max\{X,0\}$ and $X_-=\max\{-X,0\}$. 
	
	\begin{tcolorbox}[title= 2. Quantile-Based Quadrangle, halign title=center, colback=white,  colbacktitle=white, coltitle=black, fonttitle=\bfseries]
		$$
\eqalign{
  \hskip40pt
  \cS(X)=\var_\alpha(X)=q_\alpha(X)=\!\!\text{quantile}&\cr
  \hskip40pt
  \cR(X)=\cvar_\alpha(X)=\frac{1}{1-\alpha}\!\int_{\alpha}^{1}\var_\beta(X)\,d\beta
         =\!\!\text{CVaR}&\cr
  \hskip40pt
  \cD(X)=\cvar_\alpha(X-\bbe[X])
         =\!\!\text{CVaR-deviation}&\cr
  \hskip40pt
  \cV(X)=\frac{1}{1-\alpha}\,\bbe[X_+]
         =\!\!\text{partial moment, scaled}&\cr
  \hskip40pt
  \cE(X)=\bbe\!\left[\frac{\alpha}{1-\alpha}X_+ + X_-\right]
         =\!\!\text{normalized Koenker–Bassett error}&\cr
}
$$

	\end{tcolorbox}
\end{example}

We refer to \cite{MR3103448} for many more examples. We next present quadrangles not listed in \cite{MR3103448}.

\begin{example}[CVaR-Based Quadrangle, $\alpha \in (0,1)$] Estimating conditional value-at-risk (CVaR, also called superquantile) via regression has long been challenged by CVaR’s known non-elicitability property \cite{Chun2012, Gneiting2011}, which often forces indirect, quantile regression-based approximations. Overcoming this obstacle is crucial for developing more direct and reliable algorithms for superquantile estimation.

A breakthrough came with the CVaR Quadrangle developed in \cite{rockafellar2014random, CVaRRegr}, which underpins new methodologies for CVaR estimation. In particular, CVaR regression functions emerge naturally as the optimal solutions to the CVaR2 error minimization problem. This fundamental result circumvents the non-elicitability issue without relying on indirect approaches.

\emph{The CVaR-based quadrangle is regular.}

\begin{tcolorbox}[title= {3. CVaR-Based Quadrangle}, halign title=center, colback=white,  colbacktitle=white, coltitle=black, fonttitle=\bfseries]
$$
\eqalign{
  \hskip40pt
  \cS(X)=\cvar_\alpha(X)=\!\!\text{CVaR}&\cr
  \hskip40pt
  \cR(X)=\frac{1}{1-\alpha}\!\int_{\alpha}^{1}\cvar_\beta(X)\,d\beta
         =\!\!\text{CVaR2 risk}&\cr
  \hskip40pt
  \cD(X)=\frac{1}{1-\alpha}\!\int_{\alpha}^{1}\cvar_\beta(X)\,d\beta-\bbe[X]
         =\!\!\text{CVaR2 deviation}&\cr
  \hskip40pt
  \cV(X)=\frac{1}{1-\alpha}\!\int_{0}^{1}[\cvar_\beta(X)]_+\,d\beta
         =\!\!\text{CVaR2 regret}&\cr
  \hskip40pt
  \cE(X)=\frac{1}{1-\alpha}\!\int_{0}^{1}[\cvar_\beta(X)]_+\,d\beta-\bbe[X]
         =\!\!\text{CVaR2 error}&\cr
}
$$
\end{tcolorbox}
\end{example}

\begin{example}[Quantile Symmetric Average Quadrangle, $\alpha \in (0,1)$]\label{ex:cvar norm} Addressing the limitations of classical $\cL^p$-norms in functional analysis, statistical estimation, and machine learning has driven interest in alternative norms capable of quantifying the risk of rare events. One promising candidate is the CVaR norm, which provides robust performance in various regression scenarios.

Building on this motivation, the quantile symmetric average quadrangle (originally named the "CVaR Norm Quadrangle" in \cite{CVaRNorm2}) --- developed and studied in \cite{CVaRNorm2} --- introduces the CVaR norm as its error function, earlier explored in \cite{CVaRNorm, bertsimas2011robust}. This framework has shown noteworthy success in linear regression applications, where \cite{Anton2022SVR} demonstrated that regression with respect to the CVaR norm is equivalent to the well-known $\nu$-Support Vector Regression (SVR) \cite{NewSVM} and Stable Regression \cite{stable_regres}, which is based on a well-known dual formulation of CVaR.  

\emph{The quantile symmetric average quadrangle is regular.}
    \begin{tcolorbox}[title = 4. Quantile Symmetric Average Quadrangle\\ (CVaR Norm Quadrangle), halign title = center,colback=white,  colbacktitle=white, coltitle=black, fonttitle=\bfseries]
      $$
\eqalign{
  \hskip40pt
  \cS(X)=\frac{1}{2}\!\left(\var_{\frac{(1-\alpha)}{2}}(X)+\var_{\frac{(1+\alpha)}{2}}(X)\right)
         =\!\!\text{symmetric VaR}&\cr
  \hskip40pt
  \cR(X)=\frac{1}{2}\Big((1+\alpha)\cvar_{\frac{(1-\alpha)}{2}}(X)+(1-\alpha)\cvar_{\frac{(1+\alpha)}{2}}(X)\Big)
         &\cr
  \hskip40pt
  \cD(X)= \cR(X) - \bbe[X] =\!\!\text{symmetrized CVaR deviation}
&\cr
  \hskip40pt
  \cV(X)=\llangle X\rrangle_\alpha+\bbe[X]
         =\!\!\text{CVaR norm regret}&\cr
  \hskip40pt
  \cE(X)=(1-\alpha)\cvar_\alpha(|X|)
         =\llangle X\rrangle_\alpha
         =\!\!\text{CVaR norm}&\cr
}
$$

    \end{tcolorbox}
\end{example}

\begin{example}[Quantile Symmetric Average Union Quadrangle, $ 0 \leq \varepsilon < \\\frac{1}{2}(\ess(X) - \esi(X))$]\label{ex:qsauq}
Integrating machine learning tools with classical statistics and risk management has long been of interest for creating robust and versatile estimation frameworks. A notable step in this direction is the effort to embed well-known methods like 
$\varepsilon$-SVR \cite{SVR, SVR2} into the Risk Quadrangle framework.

As constructed in \cite{Anton2022SVR}, the quantile symmetric average union quadrangle accomplishes this integration: the optimal solution to the regression problem with the Vapnik error emerges as an estimate of the average of two symmetric quantiles, effectively bridging the gap between traditional statistical tools and modern machine learning methods.
Define 
$$\cA_\varepsilon(X) = \Lset \alpha \in [0,1) \Mset  \varepsilon \in \frac{1}{2}\Big(\var_{\frac{1+\alpha}{2}}(X) - \var_{\frac{1-\alpha}{2}}(X)\Big)\Rset.$$
\begin{tcolorbox}[title = 5. Quantile Symmetric Average Union Quadrangle, halign title = center,colback=white,  colbacktitle=white, coltitle=black, fonttitle=\bfseries]
   $$
{\openup2\jot
\eqalign{
  \hskip20pt
  \cS(X)=
    \bigcup\limits_{\alpha\in\cA_\varepsilon(X)}
      \tfrac{1}{2}\!\bigl(\var_{\frac{1-\alpha}{2}}(X)+\var_{\frac{1+\alpha}{2}}(X)\bigr)&\cr
  \hskip20pt
  \cR(X)=
    \tfrac{1}{2}\Bigl((1+\alpha)\cvar_{\frac{1-\alpha}{2}}(X)
      +(1-\alpha)\cvar_{\frac{1+\alpha}{2}}(X)\Bigr)&\cr
    \hskip58pt  -(1-\alpha)\varepsilon, \quad\forall\;\alpha\in\cA_\varepsilon(X)&\cr
  \hskip20pt
  \cD(X)=
    \cR(X) - \bbe[X]&\cr
  \hskip20pt
  \cV(X)=\bbe[|X|-\varepsilon]_+ + \bbe[X]=\!\!\text{Vapnik regret}&\cr
  \hskip20pt
  \cE(X)=\bbe[|X|-\varepsilon]_+=\!\!\text{Vapnik error}&\cr
}}
$$ 
\end{tcolorbox}
Note that the risk measure $\cR(X)$ in the quantile symmetric average union quadrangle is the same for all $\alpha \in \cA_\varepsilon(X).$
\emph{The quantile symmetric average union quadrangle is not regular. For instance, axiom (A9) fails for its error measure. Nonetheless, it is a subregular risk quadrangle (see Definition \ref{def:subreg quadrangle}).}
\end{example}

\begin{example}[Expectile-Based Quadrangle (Asymmetric Mean Squared Error\\ (MSE) Version), $q \in (0,1)$]
 The \emph{expectile} of a random variable $X$ at confidence level $q \in (0,1)$ is defined as (see \cite{Newey})  
\begin{equation}\label{expectile definition}
     e_q(X) = \argmin_{C\in \bbr}\left\{\cE(X-C)\right\},
\end{equation}
where $\cE(X) =  \bbe[qX_+^2 + (1-q)X_-^2]$ is the \emph{asymmetric MSE}.
The case $q = 0.5$ gives the usual mean value. 

The expectile-based quadrangle (asymmetric mean squared error version) was introduced in \cite{Expect} and studied in \cite{Anton2024expectile}. This quadrangle generalizes the mean quadrangle introduced in \cite{MR3103448}. \emph{The expectile-based quadrangle (asymmetric mean squared error version) is regular.}
\begin{tcolorbox}[title = 6. Expectile-Based Quadrangle\\ (Asymmetric Mean Squared Error Version), halign title = center,colback=white,  colbacktitle=white, coltitle=black, fonttitle=\bfseries]
$$
\eqalign{
  \hskip40pt
  \cS(X)=e_q(X)=\!\!\text{expectile}&\cr
  \hskip40pt
  \cR(X)=\cD(X)+\bbe[X]=\!\!\text{asymmetric risk}&\cr
  \hskip40pt
  \cD(X)=q\,\bbe\!\left[\bigl((X-e_q(X))_+\bigr)^2\right]
         +(1-q)\,\bbe\!\left[\bigl((X-e_q(X))_-\bigr)^2\right]
         &\cr
  \hskip40pt
  \cV(X)=\cE(X)+\bbe[X]
         =\!\!\text{asymmetric regret}&\cr
  \hskip40pt
  \cE(X)=\bbe\!\left[q\,X_+^{\,2}+(1-q)\,X_-^{\,2}\right]
         =\!\!\text{asymmetric MSE}&\cr
}
$$
\end{tcolorbox}

\end{example}

\begin{example}[Expectile-Based Quadrangle (Piecewise Linear Version), $K>0$] \label{ExpectileExample}
Despite growing interest in expectiles for risk management and statistical estimation, there has been no unified framework for both efficient optimization and statistical estimation.

Expectile can also be defined by the necessary condition of extremum for \eqref{expectile definition} as a solution to the equation
\begin{equation}\label{expectile optimality condition 1}
    q\bbe[(X-C)_+] = (1-q)\bbe[(X-C)_-].
\end{equation}
With formula $\bbe[X-C] = \bbe[(X-C)_+] - \bbe[(X-C)_-]$, equation \eqref{expectile optimality condition 1} is equivalently transformed to
\begin{equation}\label{expectile optimality condition 2}
    C - \bbe[X] = \frac{1}{K}\bbe[(X-C)_+],
\end{equation}
where $K = \frac{1-q}{2q-1}$. A one-to-one correspondence exists between values  $K>0$ and values $q$ in the interval $1/2 < q < 1$. Also, there is a one-to-one correspondence between values $K$ in the interval $K < -1$ and values $q$ in the interval $0 < q< 1/2$.

The equation \eqref{expectile optimality condition 2} should be considered separately for $K > 0  \ (1/2<q<1)$ and for $K<-1 \  (0<q<1/2)$ because properties of expectile as a function of the parameter $q$ change at the point $q = 1/2$. 

For the range  $K > 0  \ (1/2<q<1)$, expectile is a coherent risk measure in the basic sense (cf. \cite{ShapiroKusuoka}).

The expectile-based quadrangle (piecewise linear version) introduced in \cite{Expect} and studied in \cite{Anton2024expectile} has the expectile, $e_K$, as a risk measure as well as the statistic. Having the expectile as a risk allows us to leverage the \cite[Regret Theorem]{MR3103448} for efficient expectile optimization. Having a piecewise linear error allows for the reduction of the expectile estimation with linear regression to a linear programming problem. 

\emph{The expectile-based quadrangle (piecewise linear version) is regular.}
\begin{tcolorbox}[title = 7. Expectile-Based Quadrangle (Piecewise Linear Version), halign title = center,colback=white,  colbacktitle=white, coltitle=black, fonttitle=\bfseries]
       $$\eqalign{
  \hskip40pt
  \cS(X)=e_K(X)=\!\!\text{expectile}
&\cr
  \hskip40pt
  \cR(X)  = e_K(X) =\!\!\text{expectile risk} &\cr
  \hskip40pt
  \cD(X) = e_K(X - \bbe[X])
                    =\!\!\text{expectile deviation} &\cr
  \hskip40pt
  \cV(X) = \left(\bbe\left[X + \frac{1}{K}X_+\right]\right)_+ 
                    =\!\!\text{piecewise linear regret} &\cr
  \hskip40pt
  \cE(X) = \max\left\{-\bbe[X], \frac{1}{K}\bbe[X_+]\right\}
                   =\!\!\text{piecewise linear error} &}
$$ 
\end{tcolorbox}

\end{example}

\begin{example}[Mean-Based Quadrangle (Piecewise Linear Version)]\label{ex:mean_lin} Modern risk management and statistical estimation often require alternatives to the classical MSE. A piecewise linear approach can offer greater flexibility and potentially more efficient computational methods, motivating the development of new frameworks that incorporate mean–absolute risk measures within standard regression tasks. 

The mean-based quadrangle (piecewise linear version), introduced and studied by \cite{Anton2024biased}, provides exactly such a framework. By pairing a mean–absolute risk measure with the statistical estimation of mean value via regression, this quadrangle replaces the classical mean squared error with a piecewise linear alternative. Crucially, the resulting linear regression can be efficiently reformulated as a linear programming problem, offering computational advantages and more robust modeling possibilities.

\emph{The mean-based quadrangle (piecewise linear version) is regular.}
\begin{tcolorbox}[title = 8. Mean-Based Quadrangle (Piecewise Linear Version), halign title = center,colback=white,  colbacktitle=white, coltitle=black, fonttitle=\bfseries] 
  $$\eqalign{
  \hskip40pt
  \cS (X) = \bbe[X]= \!\!\text{mean}
&\cr
  \hskip40pt
  \cR (X)  = \bbe[X-\bbe[X]]_+ + \bbe[X]= \!\!\text{mean--absolute risk}
&\cr
  \hskip40pt
  \cD (X)  = \bbe[X-\bbe[X]]_+= \!\!\text{mean--absolute deviation}
  &\cr
  \hskip40pt
  \cV  (X) = \max \{\bbe[X_-], \bbe[ X_+]\} + \bbe[X]= \!\!\text{mean--absolute regret} 
  &\cr
  \hskip40pt
  \cE(X) = \max \{\bbe[X_-], \bbe[X_+]\}= \!\!\text{adjusted mean--absolute error}}
$$     
\end{tcolorbox}
 
\end{example}

\begin{example}[Biased Mean-Based Quadrangle, $x \in \bbr$] A growing need exists for alternative regression frameworks that go beyond classical mean-based approaches, particularly in engineering and finance. The concept of superexpectation (SE) \cite{rockafellar2014random} offers a way to capture biased mean estimates, which can be more relevant for certain practical applications. For instance, estimation of factors driving portfolio loss exceeding expected loss by a specified amount (e.g., $x=$ \$10 billion) or estimation of factors impacting a specific excess release of radiation in the environment (nuclear safety regulations specify different severity levels). Leveraging these ideas, the biased mean quadrangle developed in \cite{Anton2024biased} enables conditional biased mean estimation.

Within this biased mean quadrangle, the statistic is the biased mean $\cS(X)=x+\bbe[X]$, for any  $x \in \bbr.$ The superexpectation error plays a central role in estimating the conditional biased mean via regression, creating a wide range of potential engineering and financial applications. Notably, this approach to regression is equivalent to quantile regression; rather than specifying a confidence level 
$\alpha \in (0,1)$, one determines the desired quantile by choosing a distance 
$x \in \bbr$ from the mean. Furthermore, the mean-based quadrangle (piecewise linear version) in Example \ref{ex:mean_lin} emerges as a special case of this quadrangle when $x=0$.

\emph{The biased mean quadrangle is not regular. However, it is a subregular risk quadrangle, see Definition \ref{def:subreg quadrangle}.}
\begin{tcolorbox}[title = 9. Biased Mean-Based Quadrangle, halign title = center,colback=white,  colbacktitle=white, coltitle=black, fonttitle=\bfseries]
       $$\eqalign{
  \hskip40pt
  \cS (X) = x + \bbe[X]=\!\!\text{biased mean}
&\cr
  \hskip40pt
  \cR (X)  =  \bbe[X -\bbe[X]-x]_+  - x_- +\bbe[X] = \!\!\text{SE risk} 
&\cr
  \hskip40pt
  \cD (X) = \bbe[X -\bbe[X]-x]_+  - x_- = \!\!\text{SE deviation}
  &\cr
  \hskip40pt
  \cV  (X) = \max \{\bbe[X_-] - x_+, \bbe[ X_+] - x_-\} + \bbe[X] = \!\!\text{SE regret} 
  &\cr
  \hskip40pt
  \cE (X) = \max \{\bbe[X_-] - x_+, \bbe[X_+] - x_-\} = \!\!\text{SE error}}
$$ 
\end{tcolorbox}
\end{example}

\begin{example}[$\varphi$--Divergence-Based Quadrangle, $\beta>0$] The $\varphi$--divergence-based quadrangle introduced and studied in \cite{Cheng2024divergence} is based upon the concept of distributionally robust risk measures studied in \cite{shapiro2017dro, pichler}. The function $\varphi$ here is a so-called extended divergence function (we call $\varphi$ a divergence function if it additionally satisfies $ \varphi(x) = +\infty \text{for} x < 0$), i.e., a convex lower-semicontinuous function $\varphi: {\mathbb R} \to (-\infty,\infty]$ satisfying 
\begin{align}
        \varphi(1) = 0, \quad 1 \in \mathrm{int}(\{x:\varphi(x) < +\infty\}) \;.
    \end{align}
In what follows, $\varphi^*$ denotes the convex conjugate function of $\varphi.$ This quadrangle provides an interpretation of portfolio optimization, classification, and regression as robust (or distributionally robust) optimization. \emph{The $\varphi$--divergence-based quadrangle is regular.}
\begin{tcolorbox}[title = 10. $\varphi$--Divergence-Based Quadrangle, halign title = center,colback=white,  colbacktitle=white, coltitle=black, fonttitle=\bfseries]
        $$
\eqalign{
  \hskip40pt
  \cS_{\varphi,\beta}(X)=\argmin\limits_{C\in\bbr}\inf\limits_{\lambda>0}\lambda\Bigl\{\frac{C}{\lambda}+\beta+\mathbb{E}\bigl[\varphi^{*}\bigl(\tfrac{X-C}{\lambda}\bigr)-\tfrac{X}{\lambda}\bigr]\Bigr\}&\cr
  \hskip40pt
  \cR_{\varphi,\beta}(X)=\inf\limits_{\substack{C\in\bbr\\ \lambda>0}}\lambda\Bigl\{C+\beta+\mathbb{E}\bigl[\varphi^{*}\bigl(\tfrac{X}{\lambda}-C\bigr)\bigr]\Bigr\}&\cr
  \hskip40pt
  \cD_{\varphi,\beta}(X)=\inf\limits_{\substack{C\in\bbr\\ \lambda>0}}\lambda\Bigl\{C+\beta+\mathbb{E}\bigl[\varphi^{*}\bigl(\tfrac{X}{\lambda}-C\bigr)-\tfrac{X}{\lambda}\bigr]\Bigr\}&\cr
  \hskip40pt
  \cV_{\varphi,\beta}(X)=\inf\limits_{\lambda>0}\lambda\Bigl\{\beta+\mathbb{E}\bigl[\varphi^{*}\bigl(\tfrac{X}{\lambda}\bigr)\bigr]\Bigr\}&\cr
  \hskip40pt
  \cE_{\varphi,\beta}(X)=\inf\limits_{\lambda>0}\lambda\Bigl\{\beta+\mathbb{E}\bigl[\varphi^{*}\bigl(\frac{X}{\lambda}\bigr)-\frac{X}{\lambda}\bigr]\Bigr\}&\cr
}
$$
\end{tcolorbox}

\end{example}

\begin{remark}\label{rem:parent_quadrangle} The $\varphi$--divergence-based quadrangle is a special case of a broader scheme based on the generalized stochastic divergences introduced in Section~\ref{sec:gen families}. Briefly, an axiomatically defined stochastic divergence $\cJ$ gives rise (via the Fenchel--Legendre transform) to a \emph{subregular} risk measure $\cR$, for which the associated quadrangle elements $\cD$, $\cV$, $\cE$, and $\cS$ can be derived. Consequently, for $\beta>0$, one can construct a $\cJ$--divergence-based quadrangle as follows (cf.\ Proposition~\ref{prop:stocdiv}):
\begin{align*}
 \cS_{\cJ,\beta} (X) &= \argmin_{C \in \bbr}\inf_{\lambda > 0}\lambda\Big\{\beta  + \cE\bigl(\lambda^{-1}(X-C)\bigr)\Big\}\\ 
 \cR_{\cJ,\beta} (X) &= \inf_{\lambda > 0}\lambda\Big\{\beta  + \cR(\lambda^{-1}X)\Big\}\\ 
 \cD_{\cJ,\beta} (X) &=\inf_{\lambda > 0}\lambda\Big\{\beta  + \cD(\lambda^{-1}X)\Big\}\\ 
 \cV_{\cJ,\beta} (X) &=\inf_{\lambda > 0}\lambda\Big\{\beta  + \cV(\lambda^{-1}X)\Big\}\\ 
 \cE_{\cJ,\beta} (X) &=\inf_{\lambda > 0}\lambda\Big\{\beta  + \cE(\lambda^{-1}X)\Big\}
\end{align*}
In this setting, we refer to $(\cR,\cD,\cV,\cE)$ as the \emph{parent quadrangle} of the $\cJ$--divergence-based quadrangle defined above. Naturally, each element of the parent quadrangle is referred to with the prefix ``parent'' (e.g., \emph{parent risk}).     
\end{remark}

The following examples are specific instances of the $\varphi$-divergence-based quadrangle for different divergence functions $\varphi.$

\begingroup
  \addtocounter{example}{-1}
  \renewcommand\theexample{10.1}
  \begin{example}[Kullback--Leibler Divergence-Based Quadrangle, $\alpha \in (0,1)$]\label{evar quadr} 
   The Kullback--Leibler divergence-based quadrangle is built upon the entropic value-at-risk (EVaR) introduced and studied in \cite{ahmadi2012entropic}.  
The divergence function and its convex conjugate are
$$\varphi(x) = x\ln(x) -x +1, \quad \varphi^*(z) = \exp(z)-1 .$$ 
Let $\beta \coloneqq \ln\!\left(\frac{1}{1-\alpha}\right)$. The complete quadrangle can be written as follows:
\begin{tcolorbox}[title = 10.1. Kullback--Leibler Divergence-Based Quadrangle, halign title = center, colback=white, colbacktitle=white, coltitle=black, fonttitle=\bfseries]
$$
{\openup2\jot
\eqalign{
  \hskip40pt
  \cS_{\varphi,\beta}(X)=\lambda^{*}\ln\mathbb{E}\!\left[\exp\!\left(\frac{X}{\lambda^{*}}\right)\right]&\cr
  \hskip40pt
  \cR_{\varphi,\beta}(X)=\inf\limits_{\lambda>0}\lambda\Bigl\{\beta+\ln\mathbb{E}\!\Bigl[\exp\!\left(\frac{X}{\lambda}\right)\Bigr]\Bigr\}&\cr
  \hskip40pt
  \cD_{\varphi,\beta}(X)=\inf\limits_{\lambda>0}\lambda\Bigl\{\beta+\ln\mathbb{E}\!\Bigl[\exp\!\left(\frac{X-\mathbb{E}[X]}{\lambda}\right)\Bigr]\Bigr\}&\cr
  \hskip40pt
  \cV_{\varphi,\beta}(X)=\inf\limits_{\lambda>0}\lambda\Bigl\{\beta+\mathbb{E}\!\bigl[\exp\!\left(\frac{X}{\lambda}\right)-1\bigr]\Bigr\}&\cr
  \hskip40pt
  \cE_{\varphi,\beta}(X)=\inf\limits_{\lambda>0}\lambda\Bigl\{\beta+\mathbb{E}\!\Bigl[\exp\!\left(\frac{X}{\lambda}\right)-\frac{X}{\lambda}-1\Bigr]\Bigr\}&\cr
}}
$$
\end{tcolorbox}  
In the quadrangle,  $\lambda^*=\lambda^*(X)$  is a solution of the following equation:
\begin{equation*}
     \lambda^{*}\beta + \lambda^{*}\ln{\mathbb{E}\bigl[e^{\frac{X}{\lambda^{*}}}\bigl]} - \frac{\mathbb{E}\bigl[Xe^{\frac{X}{\lambda^{*}}}\bigl]}{\mathbb{E}\bigl[e^{\frac{X}{\lambda^{*}}}\bigl]} = 0.
\end{equation*}
Following the terminology of Remark~\ref{rem:parent_quadrangle}, the Kullback--Leibler divergence-based quadrangle has a parent quadrangle. This parent quadrangle is a log--exponential-based quadrangle (cf.\ \cite[Example 8]{MR3103448}). The risk measure in this quadrangle is the well-known entropic risk, $\cR(X) = \ln \bbe[\exp{X}]$. 
  \end{example}
\endgroup
\begingroup
  \addtocounter{example}{-1}
  \renewcommand\theexample{10.2}
  \begin{example}[Total Variation Divergence-Based Quadrangle, $\beta \in (0,2)$]\label{eg_tvd}
   The total variation divergence-based quadrangle relies on Example 3.10 of \cite{shapiro2017dro}, where the derivation of the risk measure was carried out. Consider the following divergence function and its convex conjugate 
$$    \varphi(x) = 
    \begin{cases}
        |x - 1|, \quad &x \geq 0 \\
        + \infty, \quad &x<0
    \end{cases}
    \text{ and }
    \varphi^*(z) =
\begin{cases}
-1 + [z + 1]_+, \quad & z \leq 1 \\
+\infty, \quad  &  z > 1
\end{cases}
\;. \label{tvd_phi}$$
The complete quadrangle is as follows:
\begin{tcolorbox}[title = 10.2. Total Variation Divergence-Based Quadrangle, halign title = center,colback=white,  colbacktitle=white, coltitle=black, fonttitle=\bfseries]
$$
{\openup2\jot
\eqalign{
  \hskip40pt
  \cS_{\varphi,\beta}(X)=\ess(X)-2\,\var_{1-\frac{\beta}{2}}(X)&\cr
  \hskip40pt
  \cR_{\varphi,\beta}(X)=\frac{\beta}{2}\,\ess(X)+(1-\tfrac{\beta}{2})\,\mathrm{CVaR}_{\frac{\beta}{2}}(X)&\cr
  \hskip40pt
  \cD_{\varphi,\beta}(X)=\frac{\beta}{2}\,\ess(X)+(1-\tfrac{\beta}{2})\,\mathrm{CVaR}_{\frac{\beta}{2}}(X)-\bbe[X]&\cr
  \hskip40pt
  \cV_{\varphi,\beta}(X)=
    \inf\limits_{\substack{\lambda>0\\ \lambda\ge\ess X}}
      \Bigl\{\,\lambda(\beta-1)+\mathbb{E}[X+\lambda]_+\,\Bigr\}&\cr
  \hskip40pt
  \cE_{\varphi,\beta}(X)=
    \inf\limits_{\lambda>0}
      \Bigl\{\,\lambda(\beta-1)+\mathbb{E}\bigl[[X+\lambda]_+ - X\bigr]\,\Bigr\}&\cr
}}
$$
 
\end{tcolorbox}
 
  \end{example}
\endgroup
\begingroup
  \addtocounter{example}{-1}
  \renewcommand\theexample{10.3}
  \begin{example}[Pearson Divergence-Based Quadrangle]\label{eg_chi_square}
      The Pearson diver\-gence-based quadrangle is a special case of the higher-order quantile-based quadrangle in Example 12 of \cite{MR3103448}. The second-order superquantile risk measure from this quadrangle was introduced and studied by Krokhmal  \cite{krokhmal2007higher_order}.
The divergence function and its convex conjugate are  
$$
     \varphi(x) = 
     \begin{cases}
      (x-1)^2,\quad &x\geq 0 \\
      +\infty,\quad &x<0
     \end{cases}
     \quad \ \textrm{ and } \ \quad  \varphi^*(z) =
     \begin{cases}
     \frac{(z+2)^2}{4}-1,\quad &z+2\geq 0 \\
     -1,\quad &z+2<0
     \end{cases}
     \;. \label{chi squared phi}
$$ 
The complete quadrangle is
\begin{tcolorbox}[title = 10.3. Pearson Divergence-Based Quadrangle, halign title = center,colback=white,  colbacktitle=white, coltitle=black, fonttitle=\bfseries]
  $$
{\openup2\jot
\eqalign{
  \hskip40pt
  \cS_{\varphi,\beta}(X)=
    \argmin\limits_{C\in\bbr}\!\Bigl(\sqrt{(\beta+1)\,\bbe\!\left[(X-C)^2_+\right]}
      -\bbe[X-C]\Bigr)&\cr
  \hskip40pt
  \cR_{\varphi,\beta}(X)=
    \min\limits_{C\in\bbr}\!\Bigl(\sqrt{(\beta+1)\,\bbe\!\left[(X-C)^2_+\right]}
      +C\Bigr)&\cr
  \hskip40pt
  \cD_{\varphi,\beta}(X)=
    \min\limits_{C\in\bbr}\!\Bigl(\sqrt{(\beta+1)\,\bbe\!\left[(X-C)^2_+\right]}
      -\bbe[X-C]\Bigr)&\cr
  \hskip40pt
  \cV_{\varphi,\beta}(X)=
    \sqrt{(\beta+1)\,\bbe\!\left[X^2_+\right]}&\cr
  \hskip40pt
  \cE_{\varphi,\beta}(X)=
    \sqrt{(\beta+1)\,\bbe\!\left[X^2_+\right]}-\bbe[X]&\cr
}}
$$

\end{tcolorbox} 
  \end{example}
\endgroup
\begingroup
  \addtocounter{example}{-1}
  \renewcommand\theexample{10.4}
  \begin{example}[Extended Pearson Divergence-Based Quadrangle]\label{eg_mean_quad}
  Consider the following extended divergence function and its convex conjugate
 \begin{equation*}
     \varphi(x) = (x-1)^2 \quad \ \textrm{and} \ \quad  \varphi^*(z) = \frac{z^2}{4} + z\;.
 \end{equation*}   
Then, the extended Pearson $\chi^2$-divergence risk measure is given by
$$
\begin{aligned}
    \cR_{\varphi,\beta}(X) &= \inf_{\lambda>0, C \in \bbr} \lambda \left\{ C + \beta + \frac{1}{4\lambda^2} \bbe[(X-C)^2] + \bbe\left[\frac{X-C}{\lambda}\right]\right\}\\
    &=  \inf_{\lambda>0, C \in \bbr} \left\{ \lambda\beta + \frac{1}{4\lambda} \bbe[(X-C)^2] + \bbe[X] \right\}\\
    & = \bbe[X] + \sqrt{\beta \mathbb{V}[X]},
\end{aligned}
$$
where $\mathbb{V}[X] = \bbe[(X-\bbe[X])^2]$ is the variance of $X$ and $(\lambda^*,C^*),$ which furnish the minimum are
$$\lambda^* = \sqrt{\frac{\mathbb{V}[X]}{4\beta}}, \qquad C^* = \bbe[X].
$$
The corresponding regret is given by 
\begin{align*}
    \cV_{\varphi,\beta}(X) &= \bbe[X] + \sqrt{\beta\bbe[X^2]}\\
    &= \bbe[X] + \sqrt{\beta}\norm{X}_2.
\end{align*}
Let $\tau = \sqrt{\beta}$ and $ \sigma(X) =\sqrt{\mathbb{V}[X]},$ then the complete quadrangle is 
\begin{tcolorbox}[title = 10.4. Extended Pearson Divergence-Based Quadrangle, halign title = center,colback=white,  colbacktitle=white, coltitle=black, fonttitle=\bfseries]
  $$
{\openup2\jot
\eqalign{
  \hskip40pt
  \cS_{\varphi,\tau}(X)=\bbe[X]=\!\!\text{mean}&\cr
  \hskip40pt
  \cR_{\varphi,\tau}(X)=\bbe[X]+\tau\,\sigma(X)=\!\!\text{safety margin tail risk}&\cr
  \hskip40pt
  \cD_{\varphi,\tau}(X)=\tau\,\sigma(X)=\!\!\text{standard deviation, scaled}&\cr
  \hskip40pt
  \cV_{\varphi,\tau}(X)=\bbe[X]+\tau\,\lVert X\rVert_{2}= \!\!\text{$\cL^2$--regret, scaled}&\cr
  \hskip40pt
  \cE_{\varphi,\tau}(X)=\tau\,\lVert X\rVert_{2}= \!\!\text{$\cL^2$--error, scaled}&\cr
}}
$$

\end{tcolorbox}
exactly the mean quadrangle in Example \ref{ex:standard}. The divergence function is the extended version of the divergence function of $\chi^2$--divergence. 
It is worth noting that the radius $\beta$ of the uncertainty set does not impact the regression result, since it only impacts the scale of the error function.

  \end{example}
\endgroup
\begingroup
  \addtocounter{example}{-1}
  \renewcommand\theexample{10.5}
  \begin{example}[Generalized Extended Pearson Divergence-Based Quadrangle]
     Let $0<q<1$. 
 Consider the following extended divergence function and its convex conjugate
$$
\varphi(x) = 
\begin{cases}
\frac{1}{1-q}(x-1)^2, \quad &x>1 \\
\frac{1}{q}(x-1)^2, \quad &x\leq 1
\end{cases}
\text{ and }
\varphi^*(z) =
\begin{cases}
\frac{(1-q)z^2}{4} + z
, \quad &z > 0\\
\frac{qz^2}{4} + z
,\quad &z \leq 0
\end{cases}
\;.
$$
The error measure is given by  
$$
\begin{aligned}
\cE_{\varphi,\beta}(X)
&=
\inf_{\lambda>0}\;
\lambda\beta + \bbe\left[
\lambda\varphi^*\left(\frac{X}{\lambda}\right) - X
\right] \\
&=
\inf_{\lambda>0}\;
\lambda\beta + 
\frac{1}{4\lambda}  \E\left[qX_{+}^2 + (1-q)X_{-}^2
\right] \\
&=
\lambda\beta + 
\frac{1}{4\lambda}  \E\left[qX_{+}^2 + (1-q)X_{-}^2
\right] \Big|_{\lambda = \sqrt{\frac{\E\left[qX_{+}^2 + (1-q)X_{-}^2 \right]
}{4 \beta}}} \\
&=
\sqrt{\beta \bbe \left[ qX_{+}^2 + (1-q)X_{-}^2 \right]}.
\end{aligned}
$$
Thus, the complete quadrangle is as follows
\begin{tcolorbox}[title = 10.5. Generalized Extended Pearson Divergence-Based\\ Quadrangle, halign title = center,colback=white,  colbacktitle=white, coltitle=black, fonttitle=\bfseries]
   $$
{\openup2\jot
\eqalign{
  \hskip10pt
  \cS_{\varphi,\beta}(X)=e_q(X)=\!\!\text{expectile}&\cr
  \hskip10pt
  \cR_{\varphi,\beta}(X)=\cD_{\varphi,\beta}(X)+\bbe[X]&\cr
  \hskip10pt
  \cD_{\varphi,\beta}(X)=
    \sqrt{\beta\Bigl(q\,\bbe\!\bigl[(X-e_q(X))_+^{\,2}\bigr]+(1-q)\,\bbe\!\bigl[(X-e_q(X))_-^{\,2}\bigr]\Bigr)}&\cr
  \hskip10pt
  \cV_{\varphi,\beta}(X)=\cE_{\varphi,\beta}(X)+\bbe[X]&\cr
  \hskip10pt
  \cE_{\varphi,\beta}(X)=
    \sqrt{\beta\,\bbe\!\bigl[q\,X_+^{\,2}+(1-q)\,X_-^{\,2}\bigr]}
    =\!\!\text{asymmetric $\cL^{2}$--error, scaled}&\cr
}}
$$

\end{tcolorbox}
Therefore, we recover the expectile-based quadrangle (asymmetric $\cL^2$--error version) introduced in \cite{Expect, Anton2024expectile}. 
The divergence function $\varphi(x)$ gives rise to a generalized Pearson $\chi^2$--divergence. 
  \end{example}
\endgroup

\section{Theoretical Framework}\label{sec:extended}
This section provides a theoretical foundation for defining, constructing, and utilizing risk quadrangles.

\subsection{Definitions and Axioms}\label{subsec: axioms and defs}
This subsection introduces the central definitions and axioms required for further development. We first fix the functional space setting and define regular measures of risk, deviation, regret, and error introduced in \cite{MR3103448} and refined in \cite{MR3357644}.


Let $(\Omega, {\cal M},\P)$ be a probability space, where $\Omega$ is a set of elementary outcomes, ${\cal M} \subseteq 2^\Omega$ is a $\sigma$-algebra of subsets of $\Omega$, and $\P$ is a probability measure on $(\Omega, {\cal M})$. A random variable (r.v.) $X: \Omega \to \bbr$ is a measurable function defined on $\Omega$ taking values in $\R$, defined up to the set of measure $0$, that is, r.v. $X$ and $Y$ such that $\P(X=Y)=1$ will be identified. The mathematical expectation of an r.v. $X$ is defined by $\E[X]=\int_\Omega X d\P$.

For $p\in[1,\infty]$, let ${\cal L}^p:={\cal L}^p(\Omega, {\cal M},\P)$ be a normed space of all r.v.'s $X$ with $\|X\|_p<\infty$, where $\|X\|_p=(\E[|X|^p])^{1/p}$ for $p<\infty$, and $\|X\|_\infty=\mathrm{ess}\sup|X|$. A functional ${\cal F}:{\cal L}^p \to (-\infty, \infty]$ is called \emph{convex} if $${\cal F}(\lambda X + (1-\lambda) Y) \leq \lambda {\cal F}(X) + (1-\lambda){\cal F}(Y) \textrm{ for all } X,Y \in {\cal L}^p  \textrm{ and every } \lambda \in [0,1]$$  

and \emph{closed (or lower-semicontinuous)} if $$\{X: {\cal F}(X)\leq C\} \textrm{ is a closed set in } {\cal L}^p \textrm{ for every constant } C.$$ 

\begin{definition}[Regular risk]\label{def:reg_risk}
  A closed convex functional $\cR: \cL^p \to (-\infty, \infty] $ is called a \emph{regular measure of risk} if
  \begin{itemize}
      \item[(R0)] $\cR(C) = C  \textrm{ for constants } C \quad \textrm{ and } \quad \cR(X) > \bbe X \textrm{ for nonconstant } X.$
  \end{itemize}
\end{definition}

\begin{definition}[Coherent risk] A closed convex functional $\cR: \cL^p \to (-\infty, \infty] $ is called a \emph{coherent measure of risk} in
the basic sense if
\begin{itemize}
    \item[(C0)] $\cR(C) = C  \textrm{ for constants } C;  
$
\item[(C1)] $\cR(\lambda X) = \lambda \cR(X)$ for all $\lambda >0;$  
\item[(C2)] $\cR(X) \leq \cR(Y)$ for all $X$ and $Y$ such that $X \leq Y$ almost surely, 
\end{itemize}
and coherent in the general sense if (C1) is left out.
\end{definition}

\begin{definition}[Regular deviation]\label{def:reg_dev}
 A closed convex functional $\cD: \cL^p \to [0, \infty] $ is called a \emph{regular measure of deviation} if
 \begin{itemize}
     \item[(D0)] $\cD(C) = 0  \textrm{ for constants } C \quad \textrm{ and } \quad \cD(X) > 0 \textrm{ for nonconstant } X.$
 \end{itemize}   
\end{definition}

\begin{definition}[Regular regret]\label{def:reg_regret}
 A closed convex functional $\cV: \cL^p \to (-\infty, \infty] $ is called a \emph{regular measure of regret} if
 \begin{itemize}
     \item[(V0)] $\cV(0) = 0 \quad \textrm{ and } \quad \cV(X) > \bbe X \textrm{ for } X \not\equiv 0.$
 \end{itemize}   
\end{definition}

\begin{definition}[Coherent regret]\label{def:coherent_regret} A closed convex functional $\cV: \cL^p \to (-\infty, \infty] $ is called a \emph{coherent measure of regret} in
the basic sense if
\begin{itemize}
    \item[(K0)] $\cV(0) = 0;$
\item[(K1)] $\cV(\lambda X) = \lambda \cV(X)$ for all $\lambda >0;$  
\item[(K2)] $\cV(X) \leq \cV(Y)$ for all $X$ and $Y$ such that $X \leq Y$ almost surely, 
\end{itemize}
and coherent in the general sense if (K1) is left out.
\end{definition}

\begin{definition}[Regular error]\label{def:reg_error}
 A closed convex functional $\cE: \cL^p \to [0, \infty] $ is called a \emph{regular measure of error} if
 \begin{itemize}
     \item[(E0)] $\cE(0) = 0 \quad \textrm{ and } \quad \cE(X) > 0 \textrm{ for } X \not\equiv 0.$
 \end{itemize}    
\end{definition}

\begin{definition}[Regular quadrangle] \label{risk quadrangle} A quartet $(\cR,\cD,\cV,\cE)$ of regular risk, deviation, regret, and error is called a \emph{regular risk quadrangle with a statistic} $\cS$ if it satisfies the following relationship formulae:
\begin{itemize}
    \item[(Q1)] \textbf{error projection:} $\cD(X)= \min\limits_C\!\lset\cE(X-C)\rset$;
    \item[(Q2)] \textbf{regret formula:} $\cR(X)= \min\limits_C\!\lset C+\cV(X-C)\rset$;
    \item[(Q3)] \textbf{mean-centering:} \begin{equation}\label{eq:riskdev}
	{\cal R}(X)={\cal D}(X)+{\mathbb E}[X], \quad {\cal D}(X)={\cal R}(X)-{\mathbb E}[X]
\end{equation}
\begin{equation}\label{eq:errregr}
	{\cal V}(X)={\cal E}(X)+{\mathbb E}[X], \quad {\cal E}(X)={\cal V}(X)-{\mathbb E}[X];
\end{equation}

\end{itemize}
where the $\argmin$ in (Q1) and the $\argmin$ in (Q2) coincide, i.e.,
\begin{itemize}
    \item[(Q4)] \textbf{statistic:} $ \cS(X) =\argmin_C\!\lset\cE(X-C)\rset
          =\argmin_C\!\lset C+\cV(X-C)\rset$.
\end{itemize}
\end{definition}
\begin{remark}[Coherent quadrangle]\label{rem:coherent_quadrangle}
 A quartet $(\cR,\cD,\cV,\cE)$ satisfying relationship formulae (Q1)--(Q4), where $\cR$ is a coherent measure of risk is called a \emph{coherent risk quadrangle with a statistic} $\cS$.   
\end{remark}
The axioms listed above are too restrictive to accommodate certain examples commonly used in the literature. One key example is the Vapnik error, defined by
\begin{equation}\label{eq:vapnik}
	{\cal E}(X)={\mathbb E}[|X|-\varepsilon]_+,
\end{equation}
where $\varepsilon\geq 0$ is a parameter. The idea behind the Vapnik error is to disregard errors whose absolute values are below $\varepsilon$ while penalizing all errors exceeding that threshold. For any $\varepsilon>0$, Vapnik error \eqref{eq:vapnik} does not satisfy (E0), because ${\cal E}(X)=0$ for r.v. $X$ such that $X=\varepsilon$ almost surely. 

This observation motivates a relaxation of the axioms listed above, leading to the following definitions.

\begin{definition}[Subregular risk]\label{def:risk}
	A closed convex functional ${\cal R}:\cL^p\to(-\infty,\\\infty]$ is called a \emph{subregular risk measure} if 
	\begin{itemize}
		\item[(R1)] ${\cal R}(C)=C$ for constants $C$ and ${\cal R}(X)\geq {\mathbb E}[X]$ for all $X$;
		\item[(R2)] for all non-constant $X \in \cL^p$ there exists $\lambda>0$ such that ${\cal R}(\lambda X) > {\mathbb E}(\lambda X)$.
	\end{itemize}
\end{definition}

\begin{definition}[Subregular deviation]\label{def:dev}
	A closed convex functional ${\cal D}:\cL^p\to[0,\infty]$ is called a \emph{subregular deviation measure} if
	\begin{itemize}
		\item[(D1)] ${\cal D}(C)=0$ for constants $C$ and ${\cal D}(X)\geq 0$ for all $X$;
		\item[(D2)] for all non-constant $X \in \cL^p$ there exists $\lambda>0$ such that ${\cal D}(\lambda X) > 0$.
	\end{itemize}
\end{definition}

\begin{definition}[Subregular regret]\label{def:regret}
	A closed convex functional ${\cal V}:\cL^p\to(-\infty,\infty]$ is called a \emph{subregular regret measure} if 
	\begin{itemize}
		\item[(V1)] $\cV(0)=0$ and ${\cal V}(X)\geq {\mathbb E}[X]$ for all $X$;
		\item[(V2)] for all non-zero $X \in \cL^p$ there exists $\lambda>0$ such that ${\cal V}(\lambda X) > {\mathbb E}(\lambda X)$.
	\end{itemize}
\end{definition}

\begin{definition}[Subregular 
 error]\label{def:error}
	A closed convex functional ${\cal E}:\cL^p\to[0,\infty]$ is called a \emph{subregular error measure} if
	\begin{itemize}
		\item[(E1)] ${\cal E}(0)=0$ and ${\cal E}(X)\geq 0$ for all $X$;
		\item[(E2)] for all non-zero $X \in \cL^p$ there exists $\lambda>0$ such that ${\cal E}(\lambda X) > 0$.
	\end{itemize}
\end{definition}

\begin{definition}[Subregular quadrangle]\label{def:subreg quadrangle} A quartet $(\cR,\cD,\cV,\cE)$ of subregular risk, deviation, regret, and error satisfying (Q1)--(Q4) is called a \emph{subregular risk quadrangle with a statistic} $\cS$.
    
\end{definition}

Axioms (E1)--(E2) relax axiom (E0) by allowing $\mathcal{E}(X) = 0$ for some nonzero but “not-too-large” random variables $X$. For example, the Vapnik error \eqref{eq:vapnik} satisfies (E1)--(E2) and thus qualifies as a subregular error measure. Similarly, axioms (V1)--(V2), (D1)--(D2), and (R1)--(R2) relax axioms (V0), (D0), and (R0), respectively.

In the following section, we will extend the theory developed in
\cite{MR3103448,MR3357644} for measures of error, regret, deviation, and risk 
to the more general subregular measures introduced in 
Definitions~\ref{def:risk}--\ref{def:error}.
In particular, we will show that relations (Q1)--(Q4) 
still hold in this broader framework.
In fact, relations (Q3) are straightforward to verify, while (Q2) is equivalent to (Q1) under (Q3), so we only need to prove (Q1).

\subsection{Fundamental Theorems and Construction of Quadrangles} \label{subsec:deviation}
This section states and proves the main theorems of the RQ framework. These theorems establish the relationships among the quadrangle elements and provide a foundation for constructing RQs.

Subsection \ref{sub: primal representation} discusses the theory in the primal space, i.e., the space where the stochastic functionals are defined. Meanwhile, Subsection \ref{sec:dual} focuses on the dual space and the aspects of conjugate duality.
\subsubsection{Primal Representation}
\label{sub: primal representation}

Rockafellar and Uryasev \cite{MR3103448} presented seven key theorems --- cornerstones of the RQ framework: \emph{Quadrangle Theorem, Scaling Theorem, Mixing Theorem, Reverting Theorem, Expectation Theorem, Regret Theorem, and  Convexity Theorem.} These theorems hold for regular quadrangles and can be found in \cite{MR3103448}.

This section revisits the aforementioned theorems for subregular quadrangles and provides proofs where the extension to subregularity is non-trivial. We begin with a central theorem of the RQ framework: 
\begin{theorem}[Quadrangle Theorem]\label{Quadrangle Theorem} Let $(\cR, \cV, \cD, \cE)$ be a subregular risk quadrangle with a statistic $\cS$. Then
   \paritem{(a)} The relations $\cD(X) = \cR(X)-\bbe X$ and $\cR(X)= \bbe X+\cD(X)$
give a one-to-one correspondence between subregular measures of risk $\cR$
and subregular measures of deviation $\cD$.  In this correspondence, $\cR$ is
positively homogeneous if and only if $\cD$ is positively homogeneous.  On the other hand, 
\begin{equation}\label{3.16}
    \text{$\cR$ is monotone \iff 
                     $\cD(X)\leq \mathrm{ess}\sup X-\bbe X\,$}
\end{equation}  
 for all $X$. 

  \paritem{(b)} The relations $\cE(X) = \cV(X)-\bbe X$ and $\,\cV(X)= \bbe X+\cE(X)$ 
give a one-to-one correspondence between subregular measures of regret $\cV$
and subregular measures of error $\cE$.  In this correspondence, $\cV$ is
positively homogeneous if and only if $\cE$ is positively homogeneous.  On
the other hand, 
\begin{equation}\label{3.17}
   \cV  \text{is monotone} \iff
                     \cE(X)\leq |\bbe X|, \  X\leq 0.
\end{equation}

    \paritem{(c)} For any subregular measure of regret $\cV$, a subregular
measure of risk $\cR$ is obtained by 
\begin{equation}\label{3.18}
      \cR(X)=\inf_C\!\Lset C+\cV(X-C) \Rset.
\end{equation}

If $\cV$ is positively homogeneous, $\cR$ is positively homogeneous.
If $\cV$ is monotone, $\cR$ is monotone.
   \paritem{(d)} For any subregular measure of error $\cE$, a subregular
measure of deviation $\cD$ is obtained by 
\begin{equation}\label{eq:devfromerr}
  \cD(X)=\inf_C\!\Lset \cE(X-C) \Rset.    
\end{equation}

If $\cE$ is positively homogeneous, $\cD$ is positively homogeneous.
If $\cE$ satisfies the condition in \eqref{3.17}, then $\cD$ satisfies the
condition in \eqref{3.16}.
   \paritem{(e)} In both (c) and (d), as long as the expression being
minimized is finite for some $C$, the set of $C$ values for which the
minimum is attained, is a nonempty closed bounded interval (typically, this interval reduces to a single point).
Moreover, when $\cV$ and $\cE$ are paired as in (b), the interval comes out 
the same and gives the associated statistic: 
\[ 
\begin{aligned}
  \cS(X) &= \argmin_C\!\lset C+\cV(X-C)\rset\\
  &=\argmin_C\!\lset\cE(X-C)\rset, 
\end{aligned}
\]
with $\cS(X+C)=\cS(X)+C.$
\end{theorem}
\begin{proof}
Parts (a) and (b) are easy to check, while part (c) follows easily from (d). Hence, we only need to prove parts (d) and (e). To do this, we first note that Lemma 2.1 in \cite{MR3357644} remains valid for subregular error measures. 

 \begin{proposition}\label{prop:lemma21}
	For a subregular error measure ${\cal E}:\cL^p\to[0,\infty]$ and a sequence $\left\{C_n\right\}_{n=1}^\infty$ of scalars,
	the following holds: If sequences $X_n \in \cL^p$ and $b_n \in {\mathbb R}$ converge to $X \in \cL^p$ and $b \in {\mathbb R}$ respectively, and ${\cal E}(X_n - C_n) \leq b_n$ for all $n$, then sequence $\left\{C_n\right\}_{n=1}^\infty$ is bounded, and any accumulation point $C_0$ satisfies ${\cal E}(X - C_0) \leq b$.
\end{proposition}
\begin{proof}
	By applying property (E2) to $X\equiv -1$ and $X\equiv 1$, we conclude that
	\begin{itemize}
		\item[(E2')] there exist constants $K_1 < 0 < K_2$ such that ${\cal E}(K_1)>0$ and ${\cal E}(K_2)>0$. 
	\end{itemize}
	We will prove the statement of the Proposition for any closed convex functional ${\cal E}:\cL^p\to[0,\infty]$ satisfying (E1) and (E2'). By contradiction, assume that sequence $\left\{C_n\right\}_{n=1}^\infty$ is unbounded. By passing to a subsequence if necessary, we may assume that all $C_n$ have the same sign (say, positive). Then we may also assume that $C_n \geq |K_1|$ for all $n$, and $\lim\limits_{n\to \infty} C_n=\infty$. Then sequence $\lambda_n=-K_1/C_n, n=1,2,\dots$ is contained in $[0,1]$ and converges to $0$. Because ${\cal E}$ is convex and ${\cal E}(0)=0$, we have ${\cal E}(\lambda Y) \leq \lambda {\cal E}(Y)$ for all $Y \in {\cal L}^p$ and all $\lambda \in [0,1]$. Hence,
	$$
	\lambda_n b_n \geq \lambda_n {\cal E}(X_n - C_n) \geq {\cal E}(\lambda_n X_n + K_1) \geq 0.
	$$ 
	Because $\lim\limits_{n\to \infty} \lambda_n b_n = \lim\limits_{n\to \infty} \lambda_n \lim\limits_{n\to \infty} b_n = 0 \cdot b = 0$ and $\lim\limits_{n\to \infty} (\lambda_n X_n + K_1) = 0 \cdot X + K_1 = K_1$, the closedness of ${\cal E}$ implies that ${\cal E}(K_1) = 0$, but this is a contradiction with (E2'). 
	
	Hence, the sequence $\left\{C_n\right\}_{n=1}^\infty$ is bounded. This implies that it must have accumulation points. If $C_0$ is any such point, then, by passing to a subsequence if necessary,  we may assume that $\lim\limits_{n\to \infty} C_n=C_0$.
	Then the closedness of ${\cal E}$ implies that 
	$$
	{\cal E}(X - C_0) = {\cal E}\left(\lim\limits_{n\to \infty} (X_n - C_n)\right) \leq \limsup\limits_{n\to \infty} {\cal E}(X_n - C_n) \leq \lim\limits_{n\to \infty} b_n = b. 
	$$ 
\end{proof}
Now, we prove statements (d) and (e). 

The proof is similar to the proof of Theorem 2.2 in \cite{MR3357644}. Let us first prove that ``inf'' in \eqref{eq:devfromerr} is always attained. If ${\cal D}(X)=\infty$, it is attained for all $C$. If ${\cal D}(X)<\infty$, then there exists a sequence $\left\{C_n\right\}_{n=1}^\infty$ such that $\lim\limits_{n\to \infty} {\cal E}(X - C_n)=D(X)$. Applying Proposition \ref{prop:lemma21} with $X_n=X$, $b_n={\cal E}(X - C_n)$ and $b={\cal D}(X)$, we obtain that sequence $C_n$ is bounded and has an accumulation point $C_0$ for which we have ${\cal E}(X - C_0) \leq b$. On the other hand, ${\cal E}(X - C_0) \geq \inf_C {\cal E}(X-C) = {\cal D}(X) = b$, hence equality holds. Thus, ``inf'' in \eqref{eq:devfromerr} is always attained, and the set of minimizers $S(X)$ is a non-empty set. Because ${\cal E}(X-C)$ is a convex, closed function of $C$, $S(X)$ is a convex and closed subset of ${\mathbb R}$, and therefore is a non-empty closed interval. For any sequence $\left\{C_n\right\}_{n=1}^\infty \subseteq S(X)$ we have ${\cal E}(X-C_n)={\cal D}(X)$, hence, by Proposition \ref{prop:lemma21} with $X_n= X$ and $b_n = {\cal D}(X)$ we conclude that $\left\{C_n\right\}_{n=1}^\infty$ is bounded. Thus, $S(X)$ is bounded. 
	
	We next prove that ${\cal D}$ in \eqref{eq:devfromerr} is a subregular deviation measure. Because ``inf'' in \eqref{eq:devfromerr} is always attained, for any $X,Y \in {\cal L}^p$ we have ${\cal D}(X)={\cal E}(X - C_X)$ and ${\cal D}(Y)={\cal E}(Y - C_Y)$ for some constants $C_X, C_Y$. Then, for any $\lambda\in[0,1]$,
    \begin{align*}
      \lambda {\cal D}(X) + (1-\lambda) {\cal D}(Y) &= \lambda {\cal E}(X-C_X) + (1-\lambda) {\cal E}(Y-C_Y)\\
      &\geq {\cal E}(\lambda X + (1-\lambda) Y - (\lambda C_X + (1-\lambda)C_Y))\\
      & \geq \inf_C {\cal E}(\lambda X + (1-\lambda) Y-C) = {\cal D}(\lambda X + (1-\lambda) Y),
    \end{align*}
	which implies that ${\cal D}$ is a convex function. We next prove its closedness. Assume that $X_n$ is a sequence converging to $X$ and ${\cal D}(X_n)\leq b < \infty$ for all $n$. Because ``inf'' in \eqref{eq:devfromerr} is always 
attained, there exists constants $C_n$ such that ${\cal E}(X_n - C_n)=
{\cal D}(X_n)\leq b$.  Then by Proposition \ref{prop:lemma21} with $b_n=b$, there exists a constant $C_0$ such that ${\cal E}(X - C_0) \leq b$. Then ${\cal D}(X) = \inf_C {\cal E}(X-C) \leq {\cal E}(X - C_0) \leq b$, which proves the closedness of ${\cal D}$.
	
	Property (D1) of ${\cal D}$ follows immediately from property (E1) of ${\cal E}$, so it is left to prove (D2). For any non-zero $Y \in {\cal L}^p$, let
	$$
		f(Y) = \sup\{\lambda \geq 0\, |\, {\cal E}(\lambda Y) = 0\}.
	$$
	Properties (E1) and (E2) imply that $0 \leq f(Y) < \infty$. Further, convexity and closedness of ${\cal E}$ in combination with (E1)  implies that ${\cal E}(\lambda Y) = 0$ for all $\lambda \in [0, f(Y)]$. Next, if a sequence $\left\{Y_n\right\}_{n=1}^\infty$ converges to $Y$ and $f(Y_n)\geq b$ for all $n$ and some constant $b$, then ${\cal E}(b Y_n) = 0$ for all $n$, which implies that ${\cal E}(b Y) = 0$, hence $f(Y) \geq b$. Thus, the function $f$ is upper-semicontinuous.
	
	Now, fix any non-constant r.v. $X \in {\cal L}^p$. Let $U$ be the set of unit vectors in ${\mathbb R}^2$. For any $u = (u_1, u_2) \in U$, let $g_X(u) := f(u_1 X - u_2)$. Because $u_1 X - u_2 \not\equiv 0$ for any $u \in U$, properties of $f$ imply that function $g_X(u)$ is non-negative, finite, and upper-semicontinuous on $U$. Because $U$ is a compact set, this implies that $B_X := \sup_{u \in U} g_X(u) < \infty$. Hence, for any constant $\mu > B_X$ and every $u=(u_1,u_2) \in U$, we have $\mu > g_X(u) = f(u_1 X - u_2)$, or equivalently ${\cal E}(\mu (u_1 X - u_2))>0$.
	
	Now, choose any $\lambda > B_X$, any $C \in {\mathbb R}$, and consider unit vector $u=(\lambda/\mu, C/\mu)$, where $\mu=\sqrt{\lambda^2+C^2}>B_X$. Then 
	$$
		0 < {\cal E}(\mu (u_1 X - u_2)) = {\cal E}(\mu ((\lambda/\mu) X - C/\mu)) = {\cal E}(\lambda X - C).
	$$
	Because $C \in {\mathbb R}$ was arbitrary, and the infimum in \eqref{eq:devfromerr} is always attained, this implies that ${\cal D}(\lambda X)>0$, and proves (D2), which completes the proof of Theorem \ref{Quadrangle Theorem}.
\end{proof}


Next, we establish the following property of the set of minimizers ${\cal S}(X)$.
\begin{proposition}\label{prop:statbounded}
	Let ${\cal E}$ be a subregular error measure. Let $\left\{X_n\right\}_{n=1}^\infty$ be the sequence of r.v.s $X_n \in {\cal L}^p$ converging to $X \in \cL^p$ such that 
	${\cal D}(X_n)\leq b$ for all $n=1,2,\dots$ and some $b \in {\mathbb R}$, where ${\cal D}$ is defined in \eqref{eq:devfromerr}. Then the union
	$$
		U = \bigcup_{n=1}^\infty {\cal S}(X_n)
	$$
	is a bounded set.
\end{proposition}
\begin{proof}
	Because ${\cal S}(X_n)$ is bounded for every $n$ by item (e) of Theorem \ref{Quadrangle Theorem}, unboundedness of $U$ would imply the existence of an unbounded sequence $\left\{C_n\right\}_{n=1}^\infty$ such that $C_n \in {\cal S}(X_n)$ for all $n$. But then ${\cal E}(X_n-C_n)={\cal D}(X_n) \leq b$. Applying Proposition \ref{prop:lemma21} with $b_n=b$, we obtain that sequence $\left\{C_n\right\}_{n=1}^\infty$ is bounded, which is a contradiction.
\end{proof}

Given a set of regular quadrangles $(\cE_i, \cV_i, \cD_i, \cR_i), \ i = 1,\ldots, r$, one can construct a new quadrangle by mixing, scaling, or reverting operations (see \cite{MR3103448}). The following theorem justifies the application of a mixing operation to subregular quadrangles.

\begin{theorem}[Mixing Theorem]\label{Mixing Theorem}
For $k=1,\ldots,r$ let $(\cR_k,\cD_k,\cV_k,\cE_k)$ be a subregular quadrangle 
quartet with statistic $\cS_k$, and consider any weights $\lambda_k>0$ with
$\lambda_1+\cdots+\lambda_r=1$.   Then a quartet 
$(\cR,\cD,\cV,\cE)$ with statistic $\cS$ given by
$$\eqalign{
	\cS(X) =\, \lambda_1\cS_1(X) +\cdots+ \lambda_r\cS_r(X), \cr
	\cR(X) = \lambda_1\cR_1(X) +\cdots+ \lambda_r\cR_r(X), \cr
	\cD(X) = \lambda_1\cD_1(X) +\cdots+ \lambda_r\cD_r(X), \cr
	{\displaystyle
		\cV(X) = \min_{C_1,\ldots,C_r}\!
		\Lset \sumn_{k=1}^r \lambda_k \cV_k(X-C_k)\Mset
		\sumn_{k=1}^r \lambda_k C_k=0\Rset, }\cr  
	{\displaystyle
		\cE(X) = \min_{C_1,\ldots,C_r}\!
		\Lset \sumn_{k=1}^r \lambda_k \cE_k(X-C_k)\Mset
		\sumn_{k=1}^r \lambda_k C_k=0\Rset.} }
$$
is a subregular quadrangle.
Moreover $(\cR,\cD,\cV,\cE)$ is monotone if every
$(\cR_k,\cD_k,\cV_k,\cE_k)$ is monotone,
and $(\cR,\cD,\cV,\cE)$ is positively homogeneous if every
$(\cR_k,\cD_k,\cV_k,\cE_k)$ is positively homogeneous.
\end{theorem}
\begin{proof}
    To prove the theorem, it suffices to verify that the equation 
    \begin{equation}\label{eq:errformix}
		{\cal E}(X)=\inf_{\substack{C_1, \dots C_r\\
		\lambda_1 C_1 + \dots + \lambda_r C_r = 0}} 
	   \sum_{k=1}^r \lambda_k {\cal E}_k(X-C_k),
	\end{equation}
defines a subregular measure of error, and that if ${\cal E}(X)<\infty$, then the set of minimizers in \eqref{eq:errformix} is a non-empty bounded closed convex subset of ${\mathbb R}^r$. The rest of the proof follows from the original Mixing Theorem in \cite{MR3103448} in combination with Theorem \ref{Quadrangle Theorem}.

Fix any $X$ with ${\cal E}(X)<\infty$. Then Theorem \ref{Quadrangle Theorem} implies that for every $k = 1, \dots, r$ set ${\cal S}_k(X)=\arg\min_C {\cal E}_k(X-C)$ is a non-empty closed bounded interval in ${\mathbb R}$, which we will write as $[J_k^-(X), J_k^+(X)]$. Note that
	$$
		f_k(C) = {\cal E}_k(X-C)
	$$
	are convex functions of $C$ with minimum on ${\cal S}_k(X)$, hence they are strictly decreasing on $(-\infty, J_k^-(X))$ and strictly increasing on $(J_k^+(X), +\infty)$.  
	Let 
	$$
	M^-(X) = \left\{(C_1, \dots, C_r) \in {\mathbb R}^r \,:\, \sum_{k=1}^r \lambda_k C_k =0, \, C_k \geq J_k^-(X), \, k=1,\dots,r \right\},
	$$
	$$
	M^+(X) = \left\{(C_1, \dots, C_r) \in {\mathbb R}^r \,:\, \sum_{k=1}^r \lambda_k C_k =0, \, C_k \leq J_k^+(X), \, k=1,\dots,r \right\},
	$$
	and $M(X) = M^-(X) \cup M^+(X)$. We claim that for every $(C_1, \dots C_r) \not \in M(X)$ satisfying $\sum_{k=1}^r \lambda_k C_k =0$ there exists $(C'_1, \dots C'_r) \in M(X)$ such that
	\begin{equation}\label{eq:objdecr}
		\sum_{k=1}^r {\cal E}_k(X-C_k) > \sum_{k=1}^r {\cal E}_k(X-C'_k).
	\end{equation}
	Indeed, $(C_1, \dots C_r) \not \in M(X)$ implies that there exists indices $i$ and $j$ such that $C_i < J_i^-(X)$ and $C_j > J_j^+(X)$. 
	Let $\delta = \min\left(\frac{J_i^-(X)-C_i}{\lambda_j}, \frac{C_j-J_j^+(X)}{\lambda_i}\right)$, and let $C'_i=C_i+\delta\lambda_j$, $C'_j=C_j-\delta\lambda_i$, and $C'_k=C_k$ for $k\not\in\{i,j\}$. Then 
	$$
	\sum_{k=1}^r \lambda_k C'_k = \lambda_i(\delta\lambda_j)+\lambda_j(-\delta\lambda_i)+\sum_{k=1}^r \lambda_k C_k=0
	$$
	and
	$$
		\sum_{k=1}^r {\cal E}_k(X-C'_k) - \sum_{k=1}^r {\cal E}_k(X-C_k) = (f_i(C'_i) - f_i(C_i)) + (f_j(C'_j) - f_j(C_j)).
	$$
	Because $J_i^-(X) \geq C'_i > C_i$ and $f_i$ is strictly decreasing on $(-\infty, J_i^-(X))$, we have $f_i(C'_i)<f_i(C_i)$. Similarly, $J_j^+(X) \leq C'_j < C_j$ implies that $f_j(C'_j)<f_j(C_j)$. This proves \eqref{eq:objdecr}. If $(C'_1, \dots C'_r) \in M(X)$, the claim is proved. Otherwise, we can repeat this process. At every step, the number of coordinates $C_k$ lying outside of the interval ${\cal S}_k(X)$ decreases. Hence, after a finite number of steps, we obtain $(C'_1, \dots C'_r) \in M(X)$, and the claim follows.
	
	The claim implies that \eqref{eq:errformix} is equivalent to
	\begin{equation}\label{eq:errformix2}
		{\cal E}(X)=\inf_{(C_1, \dots, C_r) \in M(X)} \sum_{k=1}^r
		\lambda_k {\cal E}_k(X-C_k).
	\end{equation}
 	Because $M(X)$ is a compact subset of ${\mathbb R}^r$, and the objective function is convex and lower-semicontinuous, it follows that the infimum in \eqref{eq:errformix2} is attained, and the set of minimizers is a non-empty, bounded, closed convex set.
 	
 	We next prove that ${\cal E}(\cdot)$ is a subregular error measure. Its convexity and property (E1) are obvious corollaries from the corresponding properties of ${\cal E}_k$, so we only need to prove (E2) and lower-semicontinuity. We start with (E2). Fix any non-zero $X \in {\cal L}^p$. First, assume that $X$ is non-constant. 
 	By Theorem \ref{Quadrangle Theorem} functionals 
 	$$
 		{\cal D}_k(X)=\inf_C {\cal E}_k(X-C)
 	$$
 	are subregular deviation measures, hence, by (D2), there exist positive constants $\mu_k$, $k=1,\dots, r$, such that ${\cal D}_k(\mu_k X)>0$. Then convexity of ${\cal D}_k$ together with ${\cal D}_k(0)=0$ imply that ${\cal D}_k(\mu X)>0$ for $\mu=\max_{1\leq k\leq r} \mu_k$, or equivalently, ${\cal E}_k(\mu X-C_k)>0$ for all $k$ and all constants $C_k$. But this implies that $\sum_{k=1}^r \lambda_k {\cal E}_k(\mu X-C_k) > 0$ for any positive constants $\lambda_k$. Because infimum in \eqref{eq:errformix} is always attained, this implies that ${\cal E}(\mu X)>0$ and proves (E2) for non-constant $X$.   
 	
 	If $X$ is a constant, it is sufficient to consider cases $X=\pm 1$, that is, prove property (E2') formulated in the proof of Proposition \ref{prop:lemma21}. 	
 	By property (E2') for each ${\cal E}_k$, there exist constants $K^-_k < 0 < K^+_k$ such that ${\cal E}_k(K^-_k)>0$ and ${\cal E}_k(K^+_k)>0$. If $K \in {\mathbb R}$ is any constant such that ${\cal E}(K)=0$, then there exist constants $(C_1, \dots, C_r)$ such that
 	$$
 		\sum_{k=1}^r \lambda_k C_k = 0 \quad \text{and} \quad \sum_{k=1}^r \lambda_k {\cal E}_k(K-C_k) = 0.
 	$$
 	Because all $\lambda_k>0$, this is possible only if ${\cal E}_k(K-C_k)=0$ for all $k$. But then
 	$$
 		K^-_k < K-C_k < K^+_k, \quad k=1,\dots, r.
 	$$
 	Multiplying these inequalities by $\lambda_k$ and adding, we obtain
 	$$
 		\sum_{k=1}^r \lambda_k K^-_k < K\sum_{k=1}^r \lambda_k - \sum_{k=1}^r \lambda_k C_k = K < \sum_{k=1}^r \lambda_k K^+_k  
 	$$
 	This implies that (E2') for ${\cal E}(\cdot)$ holds with constants $K^-=\sum_{k=1}^r \lambda_k K^-_k < 0$ and $K^+=\sum_{k=1}^r \lambda_k K^+_k > 0$. 
 	
 	
 	Now let us prove lower-semicontinuity of ${\cal E}(\cdot)$. We need to prove that for any sequence $X_n$ converging to $X$, we have $\liminf\limits_{n\to \infty} {\cal E}(X_n) \geq {\cal E}(X)$. By passing to a subsequence if necessary, we may assume that $L=\lim\limits_{n\to \infty} {\cal E}(X_n)$ exists. If $L=\infty$, the inequality $L \geq {\cal E}(X)$ is trivial, so we may assume that $L<\infty$. Then there is a constant $c$ such that ${\cal E}(X_n) \leq c$ for all $n$.
 	
 	Because the infimum in \eqref{eq:errformix} is attained, there exist vectors ${\bf C}_n=(C_{n1}, \dots, C_{nr})$ such that 
 	$$
 		\sum_{k=1}^r \lambda_k C_{nk} = 0 \quad \text{and} \quad {\cal E}(X_n)=\sum_{k=1}^r \lambda_k {\cal E}_k(X_n-C_{nk}), \quad n=1,2,\dots
 	$$  
 	We have proved above that ${\bf C}_n \in M(X_n)$ for all $n$. 
 	
 	Inequality $c \geq {\cal E}(X_n)$ implies that 
 	$$
 		b:= \frac{c}{\min\{\lambda_1, \dots, \lambda_r\}} \geq {\cal E}_k(X_n-C_{nk}) \geq \min_C {\cal E}_k(X_n-C), \quad k=1,\dots, r.
 	$$
 	Hence, by Proposition \ref{prop:statbounded},
 	$$
 		U_k = \bigcup_{n=1}^\infty {\cal S}_k(X_n), \quad k=1,\dots,r
 	$$
 	are bounded sets. This implies that the set
 	$$
 		M^* = \bigcup_{n=1}^\infty M(X_n)
 	$$
 	is a bounded subset of ${\mathbb R}^r$. Because ${\bf C}_n \in M^*$ for all $n$, we may assume, after passing to a subsequence if necessary, that sequence ${\bf C}_n$ converges component-wise to a vector ${\bf C}=(C_1, \dots, C_r)$. Then $\sum_{k=1}^r \lambda_k C_k = 0$. Also, for each $k$, the sequence $X_n-C_{nk}$ converges to $X-C_k$ in $\cL^p$. By lower-semicontinuity of ${\cal E}_k$, this implies that
 	$$
 	{\cal E}_k(X-C_k) \leq \liminf\limits_{n \to \infty} {\cal E}_k(X_n-C_{nk}).
 	$$
 	Then
 	$$
 	{\cal E}(X) \leq \sum_{k=1}^r \lambda_k {\cal E}_k(X-C_k) \leq \sum_{k=1}^r \lambda_k \liminf\limits_{n \to \infty} {\cal E}_k(X_n-C_{nk}) 
 	$$
 	$$
 	\leq \liminf\limits_{n \to \infty} \sum_{k=1}^r \lambda_k {\cal E}_k(X_n-C_{nk}) = \liminf\limits_{n \to \infty} {\cal E}(X_n) = L,
 	$$
    which completes the proof.
\end{proof}
 Note that the scaling and reverting theorems from \cite{MR3103448} hold for the subregular risk quadrangles. The proofs are similar to those of the mixing theorem, with an appropriate change of variables. Below, we provide the formulations of these theorems for completeness.
 
 \begin{theorem}[Reverting Theorem]\label{Reverting Theorem}
For $i=1,2$, let $(\cR_i,\cD_i,\cV_i,\cE_i)$ be a subregular quadrangle quartet 
with statistic $\cS_i$.  Then a quartet $(\cR,\cD,\cV,\cE)$ 
with statistic $\cS$ given by
$$\eqalign{
	\cS(X) = \half[\cS_1(X)-\cS_2(-X)], \cr
	\cR(X) = \bbe X+\half[\cR_1(X)+\cR_2(-X)], \cr
	\cD(X) = \half[\cD_1(X)+\cD_2(-X)] = \half[\cR_1(X)+\cR_2(-X)], \cr
	\cV(X) = \bbe X+{\displaystyle \min_C}\Lset
	\half[\cV_1(C+X)+\cV_2(C-X)]-C\Rset, \cr 
	\cE(X) = {\displaystyle \min_C}\Lset\half[\cE_1(C+X) +\cE_2(C-X)]\Rset. }
$$
is a subregular quadrangle.
Positive homogeneity is preserved in this construction, but not
monotonicity.
\end{theorem}

\begin{theorem}[Scaling Theorem]\label{Scaling Theorem}
Let $(\cR_0,\cD_0,\cV_0,\cE_0)$ be a subregular quadrangle quartet with
statistic $\cS_0$ and consider any $\lambda\in(0,\infty)$. Then a  quartet $(\cR,\cD,\cV,\cE)$ with statistic $\cS$ given by
\begin{equation}\label{eq:3.22}
\begin{aligned}
	\cS(X) &= \cS_0(X), \\
	\cR(X) &= (1-\lambda)\mathbb{E}X + \lambda \cR_0(X), \ 
	\cD(X) = \lambda \cD_0(X),\\
	\cV(X) &= (1-\lambda)\mathbb{E}X + \lambda \cV_0(X), \ 
	\cE(X) = \lambda \cE_0(X),
\end{aligned}
\end{equation}
or alternatively by
$$\eqalign{
\cS(X)=\lambda\cS_0(\lambda^{-1}X), \cr
	\cR(X) = \lambda \cR_0(\lambda^{-1}X),\quad
	\cD(X) = \lambda \cD_0(\lambda^{-1}X),\cr
	\cV(X) = \lambda \cV_0(\lambda^{-1}X),\quad
	\cE(X) = \lambda \cE_0(\lambda^{-1}X). }
$$
is a subregular 
quadrangle.
Monotonicity and positive homogeneity are preserved in these constructions,
except that monotonicity requires $\lambda\leq 1$ in \eqref{eq:3.22}.
\end{theorem}

In some examples in Section \ref{sec:examples}, the error and regret measures are given by
\begin{equation}\label{eq:experror}
	\cE(X)={\mathbb E}[e(X)]
\end{equation} 
and
\begin{equation}\label{eq:expregret}
	\cV(X)={\mathbb E}[v(X)]
\end{equation}
for some functions $e:{\mathbb R}\to[0,\infty]$ and $v:{\mathbb R}\to(-\infty,\infty]$. Then $\cE$ and $\cV$ are related by (Q3) if and only if $e$ and $v$ are related by
\begin{equation}\label{eq:evrel}
	e(x)=v(x)-x, \quad\quad v(x)=x+e(x).
\end{equation} 

\begin{theorem}[Expectation Theorem]\label{Expectation Theorem}
	For functions $e:{\mathbb R}\to[0,\infty]$ and $v:{\mathbb R}\to(-\infty,\infty]$ related by \eqref{eq:evrel}, the properties
    \begin{equation}\label{eq:eprop}
    \begin{aligned}
    &e \text{ is closed convex, }\\
    &e(0)=0, \quad e(x)\geq 0 \text{ for all } x,\\
    &e(a)>0, e(b)>0 \text{for some } a<0<b    
    \end{aligned}
	\end{equation} amount to 
    \begin{equation}\label{eq:vprop}
    \begin{aligned}
       &v \text{ is closed convex, }\\ 
       &v(0)=0, \quad v(x)\geq x \text{for all } x,\\
       &v(a)>a, v(b)>b \text{ for some } a<0<b
    \end{aligned}
	\end{equation}  and ensure that the functionals
	$$
	\cV(X)=\bbe[v(X)], \qquad  \cE(X)=\bbe[e(X)],
	$$
	form a corresponding pair consisting of a subregular measure of regret and a subregular measure of error. 
	For $X\in V=\dom\cV=\dom\cE$ let $C^\plus(X)=\sup\lset C \mset X-C\in V\rset$ 
	and $C^\minus(X)=\inf\lset C \mset X-C\in V\rset$.  The associated statistic 
	$\cS$ in the quadrangle generated from $\cV$ and $\cE$ is characterized 
	then by 
    \begin{align*}
        \cS(X)&=\Lset C\Mset \bbe[e'_\minus(X-C)]\leq 0\leq \bbe[e'_\plus(X-C)]\Rset\\ 
	&=\Lset C\Mset \bbe[v'_\minus(X-C)]\leq 1\leq \bbe[v'_\plus(X-C)]\Rset 
    \end{align*}
	subject to the modification that, in both cases, the right side is replaced 
	by $\infty$ if $C\leq C^\minus(X)$ and the left side is replaced by
	$-\infty$ if $C\geq C^\plus(X)$.  The quadrangle is completed then by 
	setting
	$$
	\cD(X)=\bbe[e(X-C)] \,\text{and}\, \cR(X) = C+\bbe[v(X-C)] \,\text{for any/all} C\in \cS(X).
	$$
	Having $\cV$ and $\cR$ be monotone corresponds (in tandem with convexity) to having $v(x)\leq 0$ when $x<0$, or equivalently $e(x)\leq |x|$ when $x<0$.  Positive homogeneity holds in the quadrangle if and only if $v$ and $e$ have graphs composed of two 
	linear pieces linked at $0$. 
\end{theorem}
\begin{proof}
	Assume that \eqref{eq:experror} is a subregular error measure. For any constant $x$, \eqref{eq:experror} implies that $\cE(x)={\mathbb E}[e(x)]=e(x)$. Hence, convexity and closedness of $\cE$ imply the corresponding properties of $e$, axiom (E1) in Definition \ref{def:error} implies that $e(0)=0$ and $e(x)\geq 0$ for all $x \in {\mathbb R}$, while axiom (E2) implies that for any $x\neq 0$ there exists $\lambda>0$ such that $e(\lambda x)>0$. Applying this property to $x=-1$ and $x=1$, we deduce that $e(a)>0$ and $e(b)>0$ for some $a<0<b$.
	
	Conversely, assume that functional $\cE(X)$ is given by \eqref{eq:experror} for some function $e:{\mathbb R}\to[0,\infty]$ satisfying \eqref{eq:eprop}. Then convexity and closedness of $\cE$ follow from the corresponding properties of $e$ and the linearity of expectation. Further, $\cE(0)={\mathbb E}[e(0)]=e(0)=0$, and, for any r.v. $X$, $\cE(X)={\mathbb E}[e(X)]\geq 0$. It is left to prove (E2). For any non-zero $X \in \cL^p$ there exists $\epsilon>0$ such that ${\mathbb P}(A)>0$, where $A=\{\omega \in \Omega: |X|\geq \epsilon\}$. Let $\lambda=\frac{\max(|a|,|b|)}{\epsilon}$. 
	Then for every $\omega \in A$ we have either $\lambda X(\omega) \leq a$ or $\lambda X(\omega) \geq b$, which by \eqref{eq:eprop} implies that $e(\lambda X(\omega))>0$. 
	Then 
	$$
		\cE(\lambda X)={\mathbb E}[e(\lambda X)] = \int_\Omega e(\lambda X(\omega))d\bbp \geq \int_A e(\lambda X(\omega))d\bbp > 0.
	$$
	The regret part follows from the error part together with relations \eqref{eq:evrel}.
\end{proof}


 Any quadrangle generated from an error measure of the form \eqref{eq:experror} will be called an \emph{expectation quadrangle}. The Quantile Symmetric Average Union Quadrangle in Example \ref{ex:qsauq} is an example of the expectation quadrangle that is not regular.

\paragraph{Construction of quadrangles.}

Any subregular error measure uniquely defines a quadrangle: the corresponding deviation measure is uniquely defined via (Q1), and then the risk and regret are uniquely defined via (Q3). The same is true for any subregular regret measure. However, any given subregular deviation measure $\cD$ belongs to infinitely many subregular quadrangles. For example, for any $\lambda>0$, functional $\cE(X)=\cD(X)+\lambda |\bbe[X]|$ is a subregular error measure whose projected deviation measure is $\cD$. Similarly, any subregular risk measure $\cR$ belongs to infinitely many subregular quadrangles.

If the risk measure $\cR$ is subregular and coherent (in the general sense), then it is easy to check that 
\begin{equation}\label{eq:seminorm}
       \cE(X) :=\cR(|X|)
\end{equation}
is a subregular error measure, which can then be used to define a quadrangle. This how the quantile symmetric average quadrangle in Example \ref{ex:cvar norm} has been constructed, starting from risk measure $\cR(X) = \cvar_\alpha(X)$. Note, however, that the risk measure in that quadrangle is not $\cvar_\alpha(X)$. 


\subsubsection{Dual Representation and Conjugate Functionals}
\label{sec:dual}
Up to this point, we have formulated the quadrangle elements as \emph{primal} convex functionals on $\cL^p$ for arbitrary $p\in[1,\infty]$. 
To derive their \emph{dual} (or \emph{envelope}) representations and geometric characterizations, we now place the theory in a standard Fenchel--Legendre framework. Such representations often yield substantial insight into the structure of these primal functionals and provide practical tools for characterizing optimality.
Specifically, we restrict to $p\in[1,\infty)$ so that the topological dual of $\cL^p$ can be identified with $\cL^q$ via the pairing 
$\langle X,Q\rangle=\E[XQ]$ (with $q\in(1,\infty]$, $\frac1p+\frac1q=1$ when $p>1$, and $q=\infty$ when $p=1$). 
In the borderline case $q=\infty$, we work with the weak-$*$ topology on $\cL^\infty$ induced by $\cL^1$; this ensures that closedness and conjugacy statements below are understood in the appropriate topology and avoids introducing dual objects beyond $\cL^q$.

Within this paired-space setting, we introduce the convex conjugate $\cF^*$ of a proper closed convex functional $\cF$ and invoke the Fenchel--Moreau representation.
The main purpose is twofold: first, to translate the subregular axioms for error/regret/deviation/risk into verifiable conditions on their respective dual representations (cf. Proposition~\ref{prop:conjugate}); and second, in the positively homogeneous case, to obtain a support-function representation $\cF(X)=\sup_{Q\in\cQ}\E[QX]$ for each quadrangle element in terms of a closed convex envelope $\cQ\subseteq\cL^q$.
This envelope viewpoint leads to a transparent geometric interpretation of subregularity in terms of (relative) quasi-interior and separation by hyperplanes, which will be used in subsequent constructions.


In this framework of paired spaces, the conjugate to a proper closed convex 
functional ${\cal F}:\cL^p\to(-\infty,\infty]$ is the functional 
${\cal F}^*:\cL^q\to(-\infty,\infty]$ given by
\begin{equation}\label{eq:conj}
	\cF^*(Q) = \sup_{X \in \cL^p}(\E[XQ]-\cF(X)), \quad \forall Q \in \cL^q.
\end{equation} 
It is well-known that ${\cal F}^*$ is proper, closed, and convex, and
\begin{equation}\label{eq:conj2}
\cF(X) = \sup_{Q \in \cL^q}(\E[XQ]-\cF^*(Q)), \quad \forall X \in \cL^p.
\end{equation} 
Let
$$
\cQ = \left\{ Q \in \cL^q \,|\, \cF^*(Q) < \infty \right\} =
   \dom \cF^*
$$
be the effective domain of $\cF^*$.

\begin{proposition}\label{prop:conjugate}
	Let ${\cal F}:\cL^p\to(-\infty,\infty]$ be a closed proper 
convex functional, and let $\cF^*$ be its conjugate. Then
	\begin{itemize}
		\item[(i)] $\cF$ is a subregular error measure if and only if $\cF^*$ satisfies:
		\begin{itemize}
			\item[(E1*)] $\cF^*(Q)\geq \cF^*(0)=0$ for every $Q \in \cL^q$;
			\item[(E2*)] for any non-zero $X \in \cL^p$ there exists $Q \in \cQ$ such that $\E[XQ]>0$.
		\end{itemize} 
		\item[(ii)] $\cF$ is a subregular regret measure if and only if $\cF^*$ satisfies:
		\begin{itemize}
			\item[(V1*)] $\cF^*(Q)\geq \cF^*(1)=0$ for every $Q \in \cL^q$;
			\item[(V2*)] for any non-zero $X\in \cL^p$ there exists $Q \in \cQ$ such that $\E[XQ]>\E[X]$.
		\end{itemize} 
		\item[(iii)] $\cF$ is a subregular deviation measure if and only if $\cF^*$ satisfies:
		\begin{itemize}
			\item[(D1*)] $\cF^*(Q)\geq \cF^*(0)=0$ for every $Q \in \cL^q$;
			\item[(D2*)] $\E[Q]=0$ for every $Q \in \cQ$;
			\item[(D3*)] for any non-constant $X\in \cL^p$ there exists $Q \in \cQ$ such that $\E[XQ]>0$.
		\end{itemize}
		\item[(iv)] $\cF$ is a subregular risk measure if and only if $\cF^*$ satisfies:
		\begin{itemize}
			\item[(R1*)] $\cF^*(Q)\geq \cF^*(1)=0$ for every $Q \in \cL^q$;
			\item[(R2*)] $\E[Q]=1$ for every $Q \in \cQ$;
			\item[(R3*)] for any non-constant $X\in \cL^p$ there exists $Q \in \cQ$ such that $\E[XQ]>\E[X]$.
		\end{itemize}
	\end{itemize}
\end{proposition}
\begin{proof}
	Let us prove (i). If $\cF$ is a subregular error measure, then, by \eqref{eq:conj}, 
	$$
		\cF^*(0) = \sup_{X \in \cL^p}(0-\cF(X)) = - \inf_{X \in \cL^p}\cF(X)= 0,
	$$
	where the last equality follows from (E1). Further, for any $Q \in \cL^q$,
	$$
		\cF^*(Q) = \sup_{X \in \cL^p}(\cE[XQ]-\cF(X)) \geq \cE[0\cdot Q]-\cF(0) = 0.
	$$
	This proves (E1*). Now, let $X \in \cL^p$ be any non-zero r.v. Then by (E2) there exists $\lambda>0$ such that ${\cal F}(\lambda X)>0$. Then \eqref{eq:conj2} implies that
	$$
		0< \cF(\lambda X) = \sup_{Q \in \cL^q}(\lambda\E[XQ]-\cF^*(Q))
	$$
	Hence, there exists $\bar Q \in \cL^q$ such that 
$\lambda\E[X\bar Q]-\cF^*(\bar Q) > 0$. 
	Because $\cF^*(\bar Q) \geq 0$, this implies that 
  $\cF^*(\bar Q) < \infty$ and $\E[X\bar Q] > 0$, and (E2*) follows.
	
	Conversely, assume that $\cF^*$ satisfies (E1*) and (E2*). Then 
$\cF$ satisfies (E1) by the argument exactly as above. To prove (E2), fix 
any non-zero $X\in \cL^p$. Then (E2*) implies that $\E[X\bar Q]>0$ 
for some $\bar Q \in \cQ$. Then \eqref{eq:conj2} implies that for any 
$\lambda > \frac{\cF^*(\bar Q)}{\E[X\bar Q]}$,
	$$
		\cF(\lambda X) = \sup_{Q \in 
  \cL^q}(\lambda\E[XQ]-\cF^*(Q)) \geq 
           \lambda\E[X\bar Q]-\cF^*(\bar Q) > 0.
	$$
	This proves (E2).
	
	Let us now prove (ii). $\cF$ is a subregular regret measure if and only if $\cE(X)=\cF(X)-\E[X]$ is a subregular error measure. Then its conjugate
	$$
		\cE^*(Q) = \sup_{X \in \cL^p}(\E[XQ]-\cE(X)) = \sup_{X \in \cL^p}(\E[X(Q+1)]-\cF(X)) = \cF^*(Q+1).
	$$
	By (i), $\cE(X)=\cF(X)-\E[X]$ is a subregular error measure if and only if $\cE^*(Q) = \cF^*(Q+1)$ satisfies (E1*) and (E2*). This happens if and only if $\cF^*(Q)$ satisfies (V1*) and (V2*).
	
	We next prove (iii). The equivalence of $\cF(X)\geq \cF(0)=0$ and (D1*) is already established when proving (i). We next prove that $\cF(C) \leq 0$ for constants $C$ if and only if (D2*) holds. Indeed,
	$$
		\cF(C) = \sup_{Q \in \cL^q}(C\E[Q]-\cF^*(Q)) \leq 0 \quad \forall C \in {\mathbb R}
	$$
	if and only if
	$$
		C\E[Q] \leq \cF^*(Q) \quad \forall Q \in \cL^q, \quad \forall C \in {\mathbb R}.
	$$
	The last inequality holds if, for every $Q \in \cL^q$, we have either $\cF^*(Q) = \infty$ or $\E[Q]=0$. But this property is exactly (D2*).
	The equivalence of (D2) for ${\cal F}$ and (D3*) for ${\cal F}^*$ can be proved exactly as the equivalence of (E2) and (E2*) in part (i), with the only difference that we start with non-constant r.v. $X\in {\cal L}^p$ instead of non-zero one. 
Finally, part (iv) follows from part (iii) in exactly the same way as part (ii) follows from (i).
\end{proof}

A functional ${\cal F}:\cL^p\to(-\infty,\infty]$ is called positively homogeneous if
\begin{itemize}
	\item[(F1)] ${\cal F}(\lambda X)=\lambda {\cal F}(X)$ for all $X \in \cL^p$ and every $\lambda\geq 0$.
\end{itemize}
Every proper closed convex functional that is positively homogeneous can be 
represented in the form
\begin{equation}\label{eq:dual}
	{\cal F}(X)=\sup_{Q \in {\cal Q}} {\mathbb E}[Q X]
\end{equation}
for some closed convex set ${\cal Q}\subseteq \cL^q$ called the risk envelope of ${\cal F}$. The risk envelope can be recovered from ${\cal F}$ by the formula
\begin{equation}\label{eq:riskenv}
	\cQ = \{ Q \in {\cal L}^q \mid {\mathbb E}[Q X] \leq {\cal F}(X) \,\, \textrm{for all}\,\, X \in {\cal L}^p \}.
\end{equation}
Then Proposition \ref{prop:conjugate} implies that:
\begin{itemize}
	\item[(E$'$)] ${\cal F}$ is a subregular error measure if and only if ${\cal Q}$ contains $0$, and, for any non-zero $X$, an r.v. $Q$ such that $\E[XQ]>0$;
	\item[(V$'$)] ${\cal F}$ is a subregular regret measure if and only if ${\cal Q}$ contains $1$, and, for any non-zero $X$, an r.v. $Q$ such that $\E[XQ]>\E[X]$;
	\item[(D$'$)] ${\cal F}$ is a subregular deviation measure if and only if ${\cal Q}$ contains $0$, we have ${\mathbb E}[Q]=0$ for every $Q \in {\cal Q}$, and, for any non-constant $X$, ${\cal Q}$ contains an r.v. $Q$ such that $\E[XQ]>0$;
	\item[(R$'$)] ${\cal F}$ is a subregular risk measure if and only if ${\cal Q}$ contains $1$, we have ${\mathbb E}[Q]=1$ for every $Q \in {\cal Q}$, and, for any non-constant $X$, ${\cal Q}$ contains an r.v. $Q$ such that $\E[XQ]>\E[X]$.
\end{itemize}

The conditions above can be equivalently reformulated in a more geometric way as follows:

\begin{itemize}
	\item[(E$'$)] ${\cal F}$ is a subregular error measure if and only if ${\cal Q}$ is a
	closed, convex subset of ${\cal L}^q$ that contains the constant $0$ in its quasi-interior; in other words, $0 \in {\cal Q}$ and every closed hyperplane $H$ containing $0$ has elements of ${\cal Q}$ in both of its associated open half-spaces;
	\item[(V$'$)] ${\cal F}$ is a subregular regret measure if and only if ${\cal Q}$ is a
	closed, convex subset of ${\cal L}^q$ that contains the constant $1$ in its quasi-interior; in other words, $1 \in {\cal Q}$ and every closed hyperplane $H$ containing $1$ has elements of ${\cal Q}$ in both of its associated open half-spaces;
	\item[(D$'$)] ${\cal F}$ is a subregular deviation measure if and only if ${\cal Q}$ is a
	closed, convex subset of the closed hyperplane $H_0=\{Q: {\mathbb E}[Q]=0\}$ in ${\cal L}^q$ that contains the constant $0$ in its quasi-interior relative to $H_0$; in other words, $0 \in {\cal Q}$
	and every closed hyperplane $H \neq H_0$ containing $0$ has elements of ${\cal Q}$ in both of its associated open half-spaces.
	\item[(R$'$)] ${\cal F}$ is a subregular risk measure if and only if ${\cal Q}$ is a
	closed, convex subset of the closed hyperplane $H_1=\{Q: {\mathbb E}[Q]=1\}$ in ${\cal L}^q$ that contains the constant $1$ in its quasi-interior relative to $H_1$; in other words, $1 \in {\cal Q}$
	and every closed hyperplane $H \neq H_1$ containing $1$ has elements of ${\cal Q}$ in both of its associated open half-spaces.
\end{itemize}


\subsection{Parent Functionals and Corresponding Positive Homogeneous Fa\-milies}\label{sec:gen families}
In the previous section, we developed dual representations for subregular functionals (risk, regret, error, and deviation). 
That is, starting from a primal functional on $\cL^p$, we constructed its dual counterpart on $\cL^q$ (via Fenchel--Legendre transformation) and characterized the dual objects axiomatically. 
We considered two regimes: the general subregular case and the subregular \emph{plus} positive homogeneous case, where the dual representation reduces to an envelope (support-function) form~\eqref{eq:dual}.

In the present section, we adopt the reverse perspective. 
We start from \emph{dual} functionals on $\cL^q$ and work in the most general setting of \emph{closed quasi-convex} functionals. 
Given such a functional $\cJ:\cL^q\to[-\infty,\infty]$, we use its sublevel sets
\[
\cQ_\tau=\{Q\in\cL^q:\cJ(Q)\le\tau\},\qquad \tau>0,
\]
as a nested family of dual sets (risk envelopes), and define the associated primal family by their support functions,
\[
\cF_\tau(X)=\sup_{Q\in\cQ_\tau}\E[XQ],\qquad X\in\cL^p.
\]
Our goal is to understand how structural properties of $\cJ$ (e.g., closedness, convexity, normalization, domain restrictions, uniqueness of minimizers) translate into properties of the family $\{\cF_\tau\}_{\tau>0}$ (e.g., monotonicity in $\tau$, concavity in $\tau$, monotonicity in $X$, subregularity in $X$), and conversely how properties of $\cF_\tau$ identify and recover the generator $\cJ$.

This construction is motivated in part by \emph{distributionally robust optimization}: one fixes a discrepancy (divergence) functional and defines an uncertainty set as a ``divergence ball'' of radius $\tau$, then optimizes a worst-case expectation, which is precisely a support function of that ball. 
Accordingly, we introduce axiomatically defined classes of stochastic discrepancy functionals--\emph{stochastic divergences} and \emph{divergence roots} --- and show how they generate families of risk/regret measures that fit naturally into the RQ framework.

In the context of the risk-quadrangle (RQ) framework, special attention is devoted to the notion of a \emph{parent functional}. 
At this point, it is helpful to include a small schematic to orient the reader.
Starting from a dual functional $\cJ$ on $\cL^q$, one constructs the nested envelopes $\cQ_\tau$
and then defines the associated family of primal functionals as the support functions of these sets, $\cF_\tau(X)$.
When $\cJ$ is \emph{convex} (in addition to being closed and proper), this construction admits an equivalent ``parent-functional'' representation: letting $\cF$ denote the convex conjugate of $\cJ$, the family $\{\cF_\tau\}_{\tau>0}$ can be recovered from $\cF$ via the perspective transform
\[
\cF_\tau(X)=\inf_{\lambda>0}\lambda\bigl[\cF(\lambda^{-1}X)+\tau\bigr], \qquad X\in\cL^p.
\]
Thus, convex dual generators $\cJ$ give rise to positive homogeneous families both through envelope (support-function) constructions and, equivalently, through a parent functional $\cF$ via the perspective transform.

Consider an arbitrary non-constant closed quasi-convex  functional 
$\cJ:\cL^q\to[-\infty,\infty]$.  Let
\begin{equation}\label{eq:minAmaxB}
	A=\inf_{Q\in \cL^q} \cJ(Q), \quad \text{and} 
            \quad B=\sup_{Q\in \cL^q} \cJ(Q).
\end{equation}
Because $\cJ$ is non-constant, $A<B$. For any $\tau \in (A,B)$, let 
\begin{equation}\label{eq:qtaudef}
	\cQ_\tau = \{Q \in \cL^q \,|\, \cJ(Q) \leq \tau \}.
\end{equation}
Then $\cQ_\tau$, $A < \tau < B$ is a family of non-empty closed convex 
proper subsets of $\cL^q$, that is nested in the sense that
\begin{itemize}
	\item[(Q1)] $Q_\tau \subseteq Q_t$ whenever $\tau \leq t$, and $\cQ_\tau=\bigcap_{t>\tau}\cQ_t$ for every $\tau>0$. 
\end{itemize}

Let $g$ be any continuous, strictly increasing function mapping the interval 
$[A,B]$ onto $[0,+\infty]$.  Then the functional 
$\cJ':\cL^q\to[0,\infty]$ defined by $\cJ'(X):=g(\cJ(X))$ produces 
via \eqref{eq:qtaudef} exactly the same family of nested closed convex sets, 
just with parameter ranging over $(0,\infty)$. Hence, if we are interested 
in studying the family $\cQ_\tau$, then without loss of generality, 
we may assume that $A=0$ and $B=\infty$, that is, 
$\cJ:\cL^q\to[0,\infty]$ and
\begin{equation}\label{eq:minzero}
	\inf_{Q\in \cL^q} \cJ(Q) = 0 \quad \text{and} \quad 
         \sup_{Q\in \cL^q} \cJ(Q) = +\infty.
\end{equation} 

Because $\cQ_\tau$ are non-empty, convex, closed sets, the formula
\begin{equation}\label{eq:dualtau}
	\cF_\tau(X)=\sup_{Q \in \cQ_\tau} {\mathbb E}[Q X] 
         = \sup_{Q: \cJ(Q) \leq \tau} {\mathbb E}[Q X], \quad \tau > 0,
\end{equation}
defines a one-parameter family of positive homogeneous, convex, closed 
functionals $\cF_\tau:\cL^p\to(-\infty,\infty]$. Because 
$\cQ_\tau$ are proper subsets of $\cL^q$, the functionals 
$\cF_\tau$ are proper, that is, not identically $+\infty$. The property 
(Q1) translates into the fact that 
\begin{itemize}
	\item[(T1)] for every fixed $X \in \cL^p$,  $\cF_\tau(X)$ 
         is a non-decreasing upper-semicontinuous function of $\tau$ on $(0,+\infty)$. 
\end{itemize}
Conversely, for any one-parameter family of positive homogeneous convex 
closed proper functionals ${\cal F}_\tau:\cL^p\to(-\infty,\infty]$, 
$\tau>0$ satisfying (T1), the corresponding risk envelopes $\cQ_\tau$, 
$\tau>0$, can be recovered by
\begin{equation}\label{eq:riskenvtau}
	\cQ_\tau = \{ Q \in \cL^q \,\big|\, {\mathbb E}[Q X] 
    \leq \cF_\tau(X) \, \text{for all}\, X \in \cL^p \},
\end{equation} 
see \eqref{eq:riskenv}, and are non-empty proper closed convex subsets of 
$\cL^q$ satisfying (Q1). This family forms level sets 
\eqref{eq:qtaudef} of the unique non-constant closed quasi-convex  
functional $\cJ:\cL^q\to[0,\infty]$ defined by
\begin{equation}\label{eq:fstardef}
\begin{aligned}
 \cJ(Q) &= \inf\{\tau>0\,|\, Q \in \cQ_\tau\} \, .
\end{aligned}
\end{equation}
Hence, we obtained the following result.

\begin{proposition}\label{prop:gencorr}
	Relations \eqref{eq:dualtau}--\eqref{eq:fstardef} define a
one-to-one correspondence between closed quasi-convex  functionals 
$\cJ:\cL^q\to[-\infty,\infty]$ satisfying \eqref{eq:minzero} and 
one-parameter families $\cF_\tau$, $\tau>0$, of positive homogeneous convex 
closed proper functionals $\cF_\tau:\cL^p\to(-\infty,\infty]$ 
satisfying (T1).
\end{proposition}

We next investigate how various properties of $\cJ$ translate into those of  
$\cF_\tau$ and vice versa.  We first note the obvious equivalence of 
the following statements.
\begin{itemize}
	\item The infimum in \eqref{eq:minzero} is attained and can be replaced by the minimum.
	\item There exists $\bar Q\in \cL^q$ such that 
$\cJ(\bar Q) = 0$.
	\item The intersection of all sets $\cQ_\tau$, $\tau>0$ is non-empty.	
	\item There exists $\bar Q\in \cL^q$ such that 
    $\cF_\tau(X)\geq {\mathbb E}[\bar Q X]$ for all $X \in \cL^p$ 
   and all $\tau>0$. 
\end{itemize}  

We next record the consequences of the uniqueness of the minimizer. The 
following statements are equivalent.
\begin{itemize}
	\item $\argmin \cJ$ is a singleton.
	\item The intersection of all sets $\cQ_\tau$, $\tau>0$ is a singleton.
	\item There exists a unique $\bar Q\in \cL^q$ such that 
$\lim\limits_{\tau\to 0+}\cF_\tau(X) = {\mathbb E}[\bar Q X]$ for all 
$X \in \cL^p$. 
\end{itemize} 

If $\bar Q$ is the unique minimizer in \eqref{eq:minzero}, 
then $\cJ(\bar Q) = 0$, and we can interpret $\cJ$ as a ``measure of 
distance'' from $Q$ to $\bar Q$. Then sets $\cQ_\tau$ can be interpreted 
as sets of r.v.s $Q$ at ``distance'' at most $\tau$ from $\bar Q$.

In particular, substituting $\bar Q=1$, we obtain the following equivalence
$$
\argmin \cJ = \{1\} \quad \Leftrightarrow \quad 
   \lim\limits_{\tau\to 0+}\cF_\tau(X) = {\mathbb E}[X] \quad 
      \text{for all} \quad X \in \cL^p,
$$
while substituting $\bar Q=0$, we obtain the following
$$
\argmin \cJ = \{0\} \quad \Leftrightarrow \quad 
   \lim\limits_{\tau\to 0+} \cF_\tau(X) = 0 \quad \text{for all} 
          \quad X \in \cL^p.
$$ 

What about the opposite limit, as $\tau\to \infty$? It is easy to see that
$$
\lim_{\tau\to\infty} \cF_\tau(X) = \sup_{Q \in \cC} {\mathbb E}[Q X], 
       \quad \text{where} \quad \cC := \cl \dom \cJ.
$$
In particular, if 
\begin{equation}\label{eq:limeqP}
   \cl \dom \cJ = \cP := \{Q \in \cL^q\,|\, Q\geq 0, \ {\mathbb E}[Q] = 1\}
\end{equation}
then
$$
\lim_{\tau\to\infty} \cF_\tau(X) = \mathrm{ess}\sup X.
$$
Also, property $Q\geq 0$ in \eqref{eq:limeqP} implies that each 
$\cF_\tau$ is monotone in sense that $\cF_\tau(X)\geq 0$ whenever 
$X\geq 0$, while property ${\mathbb E}[Q] = 1$ in \eqref{eq:limeqP} 
implies that $\cF_\tau(C)=C$ for constants $C$.

As another example, if
\begin{equation}\label{eq:cldomnonneg}
\cl \dom \cJ = \{Q \in \cL^q\,|\, Q\geq 0\}
\end{equation}
then
\begin{equation}\label{eq:limtauinf}
	\lim_{\tau\to\infty} \cF_\tau(X) = \begin{cases}
		0, & \text{if $\mathrm{ess}\sup X \leq 0$}\\
		+\infty, & \text{if $\mathrm{ess}\sup X > 0$}.
	\end{cases}
\end{equation}

We continue with the following observation.

\begin{proposition}\label{prop:corrconv}
	Let $\cJ:\cL^q\to[-\infty,\infty]$ and 
   $\cF_\tau$, $\tau>0$ be related by \eqref{eq:dualtau}--\eqref{eq:fstardef} 
  as in Proposition \ref{prop:gencorr}. Then $\cJ$ is convex if and only if 
 the family $\cF_\tau$ satisfies (T1) and
	\begin{itemize}
		\item[(T2)] for every fixed $X \in \cL^p$,  
   $\cF_\tau(X)$ is a (non-decreasing and) \emph{concave} function of 
      $\tau$ on $(0,+\infty)$. 
	\end{itemize}
	These properties are also equivalent to the property 
	\begin{itemize}
		\item[(Q2)] $(1-\lambda)Q_{\tau_1} + \lambda Q_{\tau_2} 
    \subseteq Q_\tau$ for $\tau=(1-\lambda)\tau_1 + \lambda \tau_2$ 
      for all $\tau_1>0$, $\tau_2>0$ and $\lambda\in[0,1]$
	\end{itemize}
	for the corresponding family $\cQ_\tau$, $\tau>0$.
\end{proposition} 
\begin{proof}
	By definition, concavity (T2) means that $\cF_\tau(X)\geq 
(1-\lambda)\cF_{\tau_1} + \lambda \cF_{\tau_2}$ for all $\tau_1>0$, 
$\tau_2>0$ and $\lambda\in[0,1]$, where 
$\tau=(1-\lambda)\tau_1 + \lambda \tau_2$. By \eqref{eq:riskenv}, this 
translates into (Q2). Further, with \eqref{eq:fstardef}, (Q2) is equivalent 
to the statement that if $Q=(1-\lambda)Q_1 + \lambda Q_2$ with 
$\cJ(Q_1)\leq \tau_1$ and $\cJ(Q_2)\leq \tau_2$ then $\cJ(Q)\leq \tau$, 
where $\tau=(1-\lambda)\tau_1 + \lambda \tau_2$. But this is exactly the 
convexity of $\cJ$. 
\end{proof}

We remark that for non-constant convex functionals, the supremum condition 
in \eqref{eq:minzero} is automatically satisfied, hence \eqref{eq:minzero} 
reduces to the infimum condition.

\begin{definition}[Divergence root]
	A closed convex functional $\cK:\cL^q\to[0,\infty]$ 
satisfying
$$\inf_{Q\in \cL^q} \cK(Q) = 0 \quad \text{and} \quad 
         \sup_{Q\in \cL^q} \cK(Q) = +\infty
$$
such that $\argmin \cK = \{1\}$ and 
$\cl \dom \cK = \{Q \in \cL^q\,|\, Q\geq 0\}$ is called a 
\emph{divergence root}.
\end{definition}

The discussion above motivates the following statement.


\begin{proposition}\label{prop:divroot}
	If a functional $\cK:\cL^q\to[0,\infty]$ is a divergence 
root, then the corresponding family of positive homogeneous closed convex 
functionals $\cF_\tau$ in \eqref{eq:dualtau} with $\cJ = \cK$ satisfies (T1), (T2), 
$\lim\limits_{\tau\to 0+}\cF_\tau(X) = {\mathbb E}[X]$, \eqref{eq:limtauinf}, 
and each $\cF_\tau$ is monotone.  In particular, $\cF_\tau$ is a subregular 
regret measure for every $\tau>0$.
\end{proposition} 
\begin{proof}
	Only the last statement is new and requires proof. The inequality 
$\cF_\tau(X) \geq {\mathbb E}[X]$ is obvious from (T1) and 
$\lim\limits_{\tau\to 0+}\cF_\tau(X) = {\mathbb E}[X]$. Hence, (V1) in 
Definition \ref{def:regret} holds. If (V2) fails, then there exists a 
non-zero $X$ such that $\cF_\tau(X) = {\mathbb E}[X]$. But then (T1) and 
(T2) imply that $\cF_\tau(X) = {\mathbb E}[X]$ for all $\tau>0$, which 
contradicts \eqref{eq:limtauinf}.
\end{proof}

\begin{definition}[Stochastic divergence]
	A closed convex functional $\cJ:\cL^q\to[0,\infty]$ 
satisfying \eqref{eq:minzero} such that $\argmin \cJ = \{1\}$ and 
$\cl \dom \cJ = \cP$ is called a \emph{stochastic divergence}.
\end{definition}

The discussion above suggests the following result.

\begin{proposition}\label{prop:stocdiv}
	If functional $\cJ:\cL^q\to[0,\infty]$ is a stochastic 
divergence, then the corresponding family of positive homogeneous convex 
closed functionals $\cF_\tau$ in \eqref{eq:dualtau} satisfies (T1), (T2), 
$\lim\limits_{\tau\to 0+}\cF_\tau(X) = {\mathbb E}[X]$, 
$\lim\limits_{\tau\to\infty} \cF_\tau(X) = \mathrm{ess} \sup X$, each $\cF_\tau$ is 
monotone and satisfies $\cF_\tau(C)=C$ for constants $C$. In particular, 
$\cF_\tau$ is a subregular (in fact, a coherent) risk measure for 
every $\tau>0$.
\end{proposition} 
\begin{proof}
	Only the last statement is new and requires proof. The inequality 
$\cF_\tau(X) \geq {\mathbb E}[X]$ is obvious from (T1) and 
$\lim\limits_{\tau\to 0+}{\cal F}_\tau(X) = {\mathbb E}[X]$. Hence, (R1) 
in Definition \ref{def:risk} holds. If (R2) fails, then there exists a 
non-constant $X$ such that ${\cal F}_\tau(X) = {\mathbb E}[X]$. But then 
(T1) and (T2) imply that $\cF_\tau(X) = {\mathbb E}[X]$ for all $\tau>0$, 
which implies that ${\mathbb E}[X] = 
\lim\limits_{\tau\to\infty} \cF_\tau(X) = \mathrm{ess} \sup X$. But this is possible 
only if $X$ is a constant.
\end{proof}

\paragraph{Stochastic divergence as a measure of distance} A stochastic divergence $\cJ$ can also be interpreted as a ``distance'' between 
\emph{probability measures}. Recall that all random variables are defined on the probability space $(\Omega,\mathcal{M},\mathbb{P})$, with $\mathbb{P}$ taken as the reference measure. Hence, the notation $\mathbb{E}[Q]$ is understood as shorthand for $\mathbb{E}_{\mathbb{P}}[Q]$. Any r.v. $Q \in \cP$ satisfies $Q\geq 0$ and 
${\mathbb E}[Q] = 1$. Hence, we can define the probability measure $\P_Q$ as 
$\P_Q(A)={\mathbb E}_\P[Q I_A]$ for any event $A \in \cM$, where $I_A$ is the 
indicator function. Then $\bbe_{\P_Q}[X]=\bbe[QX]$ for any r.v. $X$. Hence, for 
any $\cQ\subseteq \cP$, 
$$
\sup_{Q \in \cQ} {\mathbb E}[Q X] = \sup_{\P_Q \in \mathfrak{P}_{\cQ}} {\mathbb E}_{\P_Q}[X],
$$
where $\mathfrak{P}_{\cQ}$ is the set of probability measures corresponding to r.v.s 
$Q \in \cQ$. Moreover, by the Radon--Nikodym theorem, there is a one-to-one correspondence between probability measures $\bbp_Q  \ll \bbp$ and \emph{densities} $Q = \bbp_Q / \bbp$, unique $\bbp$-a.s.

In particular, \eqref{eq:dualtau} reduces to
\begin{equation}\label{eq:dualtau2}
	\cF_\tau(X)=\sup_{\P_Q \in \mathfrak{P}_{\cQ_\tau}} {\mathbb E}_{\P_Q}[X], 
\end{equation}
where $\mathfrak{P}_{\cQ_\tau} = \{\P_Q: \cJ(Q)\leq \tau\} 
= \{\P_Q: \cG(\P_Q)\leq \tau\}$ with $\cG(\P_Q) := \cJ(Q)$. The functional $\cG$ maps probability measures 
absolutely continuous with respect to $\P$ to $[0,\infty]$. It is convex, 
closed with $\cG(\P)=0$, and, for $\P_Q\neq \P$, $\cG(\P_Q)$ can be interpreted 
as the ``distance'' from $\P_Q$ to $\P$. By Proposition \ref{prop:stocdiv}, 
$\cF_\tau$ is a coherent risk measure in the basic sense for every such 
$\cG$ and every $\tau>0$.

A popular example of such ``distance'' is the Wasserstein divergence
\[
W(\P_1, \P_2) = \inf_{\gamma \in \Pi(\P_1, \P_2)} \int_{\Omega \times \Omega} d(x, y) \, d\gamma(x, y)
\]
where:
\begin{itemize}
    \item \( \Pi(\P_1, \P_2) \) is the set of all joint probability measures \( \gamma \) on \( \Omega \times \Omega \) with marginals \( \P_1 \) and \( \P_2 \);
    \item \( d(x, y) \) is a metric on $\Omega$.
\end{itemize}
A family of risk measures \eqref{eq:dualtau2} with ${\cal G}(\P_Q) = W(\P_Q,\P)$ is discussed in \cite{Rockafellar2024}.

We next consider the case when
\begin{equation}\label{eq:varphidivdef}
\cJ(Q) = {\mathbb E}(\varphi(Q)), \quad Q \in \cL^q
\end{equation}
and investigate under what conditions on $\varphi:{\mathbb R}\to(-\infty,\infty]$ functional \eqref{eq:varphidivdef} has the properties listed above. We have the following implications:
\begin{itemize}
	\item If $\varphi$ is closed proper convex, then so is $\cJ$;
	\item If $\varphi$ is non-constant, then so is $\cJ$;
	\item If $\inf_{y \in {\mathbb R}} \varphi(y)=0$ and 
$\sup_{y \in {\mathbb R}} \varphi(y)=\infty$, then  \eqref{eq:minzero} holds;
	\item In particular, if $\varphi$ is convex, non-constant, and 
$\inf_{y \in {\mathbb R}} \varphi(y)=0$,  then \eqref{eq:minzero} holds;
	\item If the infimum $\inf_{y \in {\mathbb R}}\varphi(y)$ is attained and can be replaced by a minimum, then the same is true for the infimum in \eqref{eq:minzero};
	\item If $\argmin \varphi=\{C\}$ is a singleton, then 
 $\argmin \cJ=\{C\}$; 
	\item In particular, $\argmin \varphi=\{0\} \Rightarrow \argmin \cJ
=\{0\}$ and $\argmin \varphi=\{1\} \Rightarrow \argmin \cJ=\{1\}$;
	\item If $\varphi$ is convex, then $\cl \dom \varphi=[a,b]$ is a 
closed convex interval, where $-\infty\leq a<b\leq \infty$. Then 
$\cl \dom \cJ=\{Q \in \cL^q\,|\, a\leq Q\leq b\}$;
	\item In particular, if $\cl \dom \varphi=[0,\infty]$, then 
\eqref{eq:cldomnonneg} holds;
	\item If $\varphi$ is convex, closed, $\min_{y \in {\mathbb R}} \varphi(y)=0$, $\argmin \varphi=\{1\}$, 
	and $\cl \dom \varphi=[0,\infty]$, then $\cJ$ is a divergence root (i.e., $\cJ = \cK$).
\end{itemize}


Assume that $\cJ$ is convex, and let $\cF$ be the conjugate functional to 
$\cJ$ defined in \eqref{eq:conj2}.  Then applying the perspective transform to $\cF$ yields (cf. \cite{Rockafellar2024}) 
\begin{equation}\label{eq:basic}
	\cF_\tau(X)=\inf_{\lambda>0} \lambda[\cF(\lambda^{-1}X)+\tau].
\end{equation}

\begin{proposition}\label{prop:families}
	If $\cF$ is a subregular error, regret, deviation, or risk measure, 
then $\cF_\tau$ is a positive homogeneous subregular error, regret, deviation, or risk measure, respectively.
\end{proposition}
\begin{proof}
	If $\cF$ is, for example, an error measure, then $\cF^*$ satisfy the 
conditions (E1*) and (E2*) in Proposition \ref{prop:conjugate}. Then (E1*) 
implies \eqref{eq:minzero}, while (E2*) implies that for any non-zero 
r.v. $X$ there exists $\bar Q$ such that $\cF^*(\bar Q)<\infty$ and  
$0<\E[X\bar Q]$.  Let $\lambda=\min\{\frac{\tau}{\cF^*(\bar Q)},1\}$ if 
$\cF^*(\bar Q)>0$ and $\lambda=1$ if $\cF^*(\bar Q)=0$. Then the convexity of 
$\cF^*$ implies that
	$$
	\cF^*(\lambda\bar Q) = \cF^*(\lambda\bar Q + (1-\lambda)\cdot 0) \leq 
         \lambda \cF^*(\bar Q) + (1-\lambda)\cdot 0 \leq \tau. 
	$$ 
	Then for $Q'=\lambda\bar Q$ we have $Q' \in \cQ_\tau$ and 
$0<\E[X Q']$.  Also, $\cQ_\tau$ contains $0$. Hence, it satisfies the 
condition (E') in Section \ref{sec:dual}, which implies that it is a dual 
set of an error measure. The proofs for the regret, deviation, and risk 
measures are similar.
\end{proof}

If \eqref{eq:basic} holds, we will call $\cF$ the \emph{parent} functional for a family $\cF_\tau$, $\tau>0$. In particular, if $\cK:\cL^q\to[0,\infty]$ is a divergence root, then its conjugate $\cF$ is a subregular regret measure, which is the parent of the family $\cF_\tau$ of subregular regret measures corresponding to $\cK$ via Proposition \ref{prop:divroot}.

\begin{proposition}\label{prop:projfam}
	Let $(\cE, \cV,\cD,\cR)$ be a quadrangle generated starting from a
subregular error measure $\cE$ (or from subregular regret measure $\cV$) via 
relations (Q2), (Q3).
	Let $\tau > 0$ and let $\cE_\tau$, $\cV_\tau$, $\cD_\tau$, and $\cR_\tau$ be positive homogeneous subregular error, regret, deviation, and risk measures defined in \eqref{eq:basic}. Then $\cE_\tau$, $\cV_\tau$, $\cD_\tau$ and $\cR_\tau$ are also related by (Q2), (Q3). 
\end{proposition}
\begin{proof}
	We first prove \eqref{eq:errregr}. Indeed,
	$$
		\cV_\tau(X) = \inf_{\lambda>0} \lambda[\cV(\lambda^{-1}X)+\tau] = \inf_{\lambda>0} \lambda[\cE(\lambda^{-1}X)+\E[\lambda^{-1}X]+\tau] 
	$$
	$$
		= \E[X] + \inf_{\lambda>0} \lambda[\cE(\lambda^{-1}X)+\tau] = \E[X]+\cE_\tau(X).
	$$
	The proof of \eqref{eq:riskdev} is similar. Next,
    \begin{align*}
        \cD_\tau(X) = \inf_{\lambda>0} [\lambda \cD(\lambda^{-1}X)+\lambda \tau] &= \inf_{\lambda>0} [\lambda \inf_{C_1}\cE(\lambda^{-1}X-C_1)+\lambda \tau]\\
       & = \inf_{\lambda>0} \inf_{C}[\lambda \cE(\lambda^{-1}(X-C))+\lambda \tau],
    \end{align*}
	where $C=\lambda C_1$. We can then interchange the order of infimums and obtain
	$$
		\cD_\tau(X) = \inf_{C} \inf_{\lambda>0} [\lambda \cE(\lambda^{-1}(X-C))+\lambda \tau] = \inf_{C} \cE_\tau(X-C).
	$$
	This proves (Q1), and (Q2) can be proved similarly.
\end{proof}

As a simple example, if we start with an error measure
$$
\cE(X) = \E[X^2], 
$$
then, for any $\tau>0$,
$$
	\cE_\tau(X)=\inf_{\lambda>0} \lambda[\E[(\lambda^{-1}X)^2]+\tau] 
       = \inf_{\lambda>0} [\lambda^{-1}\E[X^2]+\lambda\tau] = 
         2\sqrt{\tau}\sqrt{\E[X^2]} = 2\sqrt{\tau}\|X\|_2.
$$
The projected deviation measure is then
$$
	\cD_\tau(X) = 2\sqrt{\tau}\sigma(X),
$$
see Example 1 in \cite{MR3103448}. Equivalently, we may first note that the 
projected deviation measure for $\cE$ is 
$$
	\cD(X) = \sigma^2(X),
$$
see Example 2 in \cite{MR3103448}, and then by \eqref{eq:basic}
$$
	\cD_\tau(X)=\inf_{\lambda>0} \lambda[\sigma^2(\lambda^{-1}X)+\tau] 
     = 2\sqrt{\tau}\sigma(X).
$$

As another example, \cite{Rockafellar2024} proved that if $\cV$ is the indicator of the set 
$$
	\cX = \{X | \, \exists Y \geq 1, \, \text{such that}\, \E[X+Y]=1, \, X+Y \geq 0\}
$$
then the corresponding risk measure $\cR_\tau$ is the conditional value-at-risk 
with $\tau=\frac{\alpha}{1-\alpha}$. We remark that the set $\cX$ can be 
written in a simpler form. Indeed, with $Z=Y-1$ we obtain
$$
\cX = \{X | \, \exists Z \geq 0, \, \text{such that}\, \E[X+Z]=0, \, X+Z \geq -1\}.
$$
We claim that $X \in \cX$ if and only if $\E[\max(-1,X)]\leq 0$. Indeed, $Z \geq 0$ implies that $X+Z\geq X$. Hence, $X+Z \geq \max(-1,X)$, and $\E[\max(-1,X)] \leq \E[X+Z]=0$. On the other hand, if $\E[\max(-1,X)]\leq 0$, then there exist a constant $c>0$ such that $\E[\max(-1+c,X+c)]=0$, and we may take $Z=\max(-1+c,X+c)-X$. In conclusion,
$$
	\cX = \{X | \, \E[\max(-1,X)]\leq 0\} = \{X | \,  \E[X+1]_+ \leq 1\},
$$ 
and
$$
	\cV(X)=\begin{cases}
		0, & \text{if } \E[X+1]_+ \leq 1\\
		+\infty, & \text{otherwise.}
	\end{cases}
$$
This functional is a subregular regret measure. Indeed, inequality $\cV(X)\geq \E[X]$ is trivial if $\cV(X)=\infty$. If, conversely, $\cV(X)=0$, then $1 \geq \E[X+1]_+ \geq \E[X+1]$ implies that $0 \geq \E[X]$, and (V1) follows. From this argument, it is clear that equality ${\cal V}(X)={\mathbb E}[X]$ holds if and only if $X$ belongs to the set $A=\{X: \E[X]=0, \, \P(X\geq -1)=1\}$. Obviously, for any non-zero $X \in A$ there exists $\lambda > 0$ such that $\lambda X \not \in A$. This implies (V2).
The corresponding risk measure is
$$
\cR(X) = \inf \{C \, | \, \E[X+1-C]_+ \leq 1 \}.
$$
Also, for any $\tau>0$,
$$
	{\cal V}_\tau(X)=\inf_{\lambda>0} \lambda[\cV(\lambda^{-1}X)+\tau] = \tau \inf\{\lambda>0 : \E[\lambda^{-1}X+1]_+ \leq 1\}.
$$
\section{Generalized Regression and Statistical Estimation}\label{sec: Regression}

\subsection{Functional Regression}

Regression is one of the central concepts in statistical estimation theory. Given a random variable $Y \in \cL^p$ (\emph{the regressant} or independent variable) and a collection of random variables $X_i \in \cL^p, \ i = 1,\ldots,n,$ (\emph{the regressors} or independent variables), the task of \emph{functional regression} (see \cite{KendallP1, KendallP2}) is to find a function $f: \bbr^n \to \bbr,$ belonging to a class of measurable functions $\cC,$ that minimizes the \emph{regression residual} $Z_f:= Y - f(\mathbf{X}), \ \mathbf{X} = (X_1,\ldots, X_n),$ with respect to a particular error $\cE$ (e.g., mean squared error, mean absolute error). Specifically, the goal is to solve the following stochastic optimization problem:
\begin{equation}\label{regression problem}
    \min_{f \in \cC} \quad \cE(Z_f).
\end{equation}
In general, different choices of error result in different optimal solutions of \eqref{regression problem}.

From the statistical estimation perspective, given the regressant $Y$ and the vector of regressors $\mathbf{X}$, the aim is to estimate (track) a desired \emph{conditional statistic} $\cS(Y|\mathbf{X})$ (e.g., conditional mean $\bbe[Y|\mathbf{X}]$ or conditional quantile $\var_\alpha(Y|\mathbf{X})$) via regression. The classical approach to this problem is to find an appropriate \emph{loss function} $\ell: \bbr \to \bbr$ and solve \eqref{regression problem}, where $\cE(Z_f) = \bbe[\ell(Z_f)]$ (such errors are referred to as the errors of expectation type according to \cite{MR3103448}). Then (cf. \cite{Bach})
\begin{equation}
    f^*(\mathbf{x}) \in \argmin_{C \in \bbr} \quad \bbe[\ell(Y-C)|\mathbf{X} = \mathbf{x}],
\end{equation}
where the equality $\mathbf{X} = \mathbf{x}$ is understood pointwise.

Of course, the above approach works provided a loss function exists for a given statistic. Statistics for which such a loss function exists are called \emph{elicitable} (see \cite{Lambert2008}). For non-elicitable statistics, however, the expected loss approach is infeasible, and thus other approaches should be considered. The RQ provides a unified framework for both elicitable and non-elicitable statistics by considering axiomatically defined errors beyond expected losses.
\begin{remark}[Maximum likelihood]A prominent special case of the expectation-type error scheme is the \emph{maximum likelihood} approach.
Assume that the regression residual $Z_f:=Y-f(\mathbf X)$ admits a (conditional) density $\rho$ under the data-generating model.
Maximizing the likelihood over $f\in\cC$ is equivalent to minimizing the negative log-likelihood,
\[
\min_{f\in\cC}\ \E\!\left[-\ln \rho(Z_f)\right]
=\min_{f\in\cC}\ \E\!\left[\ell(Z_f)\right],
\qquad \ell(z):=-\ln \rho(z),
\]
which is exactly an error of expectation type.
If the residual density is specified up to parameters (say $\rho_\eta$), one may estimate $(f,\eta)$ jointly by minimizing
$\E[-\ln \rho_\eta(Z_f)]$. 
\end{remark}

\begin{theorem}[Regression Theorem]\label{Regression Theorem}
    Consider problem 
    \begin{equation}\label{5.6}
        \text{minimize} \cE(Z_f) \text{over all} f\in \cC,
    \end{equation} 
    where $Z_f=Y-f(\mathbf{X})$ is the regression residual for random variables $\mathbf{X}$ and $Y$ in the case of $\cE$ being a subregular measure of error and $\cC$ being a class of 
functions $f:\bbr^n \to \bbr$ such that 
\begin{equation}\label{5.8}
     f\in \cC \implies f+C \in \cC \text{for all} C\in \bbr.
\end{equation}

Let $\cD$ and $\cS$ correspond to $\cE$ as in the Quadrangle Theorem.
Problem \eqref{5.6} is equivalent then to:
\begin{equation}\label{5.9}
  \text{minimize} \cD(Z_f) \text{over all} f\in\cC \text{s.t.}
                   0\in \cS(Z_f). 
\end{equation}  

Moreover if $\cE$ is of expectation type and $\cC$ includes a function $f$ 
satisfying 
\begin{equation}\label{5.11}
\begin{aligned}
    f(\mathbf{x})\in \cS(Y|\mathbf{x}) 
       \text{almost surely for $\mathbf{x} \in D$,}\\
   \!\text{where} (Y|\mathbf{x}) = (Y|\mathbf{X}=\mathbf{x})      \text{(conditional distribution), }
\end{aligned}
\end{equation}
with $D$ being the support of the distribution in $\bbr^n$ induced by 
$\mathbf{X}$, then $f$ solves the regression problem \eqref{5.6} and tracks this conditional statistic
in the sense that 
\begin{equation}\label{stat tracking}
     f(\mathbf{X}) \in \cS(Y|\mathbf{X}) \text{almost surely.}
\end{equation}
\end{theorem}

The Regression Theorem \ref{Regression Theorem} of \cite{MR3103448} remains valid for the subregular functionals with the same proof.

In general, the inclusion \eqref{stat tracking} holds only for the errors of expectation type. However, more can be said in the case of linear regression.

\subsection{Linear Regression}


Consider the linear regression problem
\begin{equation}\label{eq:linrerg}
	\min\limits_{(c_0,c_1,\dots,c_n)\in {\mathbb R}^{n+1}} {\cal E}\left(Y-c_0-\sum_{i=1}^n c_i X_i\right),
\end{equation}
where ${\cal E}$ is a subregular error measure. Theorem 3.1 in \cite{MR2442649} proves that the solution set in \eqref{eq:linrerg} is non-empty under some additional assumptions on ${\cal E}$ such as positive homogeneity. Here we prove the same result for any subregular error measure with no additional assumptions.

\begin{proposition}\label{prop:linreg}
	Let ${\cal E}$ be a subregular error measure. 
	Then the set of minimizers in \eqref{eq:linrerg} is a non-empty, closed, convex subset of ${\mathbb R}^{n+1}$. 
\end{proposition} 
\begin{proof}
	The convexity and closedness of the set of minimizers follow from the convexity and lower-semicontinuity of the objective function, so we only need to prove its non-emptiness. 
	Let ${\cal X}$ be the set of all r.v.s $X$ representable as $X=c_0+\sum_{i=1}^n c_i X_i$ for some $(c_0,c_1,\dots,c_n)\in {\mathbb R}^{n+1}$. Optimization problem \eqref{eq:linrerg} can be rewritten as
	$$
		\min\limits_{X \in {\cal X}} {\cal E}\left(Y-X\right),
	$$
	and, in this formulation, it is clear that 
	we may assume that the random variables $X_1, \dots, X_n$ satisfy the linear independence condition that $\sum_{i=1}^n c_i X_i$ is not constant unless $c_1=\dots=c_n=0$, because otherwise we can remove some of the $X_i$ without changing set ${\cal X}$.
	
	Function
	$$
		f(d, c_0, c_1, \dots, c_n) = {\cal E}\left(dY-c_0-\sum_{i=1}^n c_i X_i\right)
	$$
	is a convex lower-semicontinuous function on ${\mathbb R}^{n+2}$, satisfying $f(x)\geq 0$ for all $x \in {\mathbb R}^{n+2}$ and $f(0)=0$. We need to minimize $f$ subject to the constraint that $d=1$. If $f(1,c_0, c_1, \dots, c_n)$ is identically $+\infty$ then the statement of the Proposition trivially holds. Otherwise select some $c_0, c_1, \dots, c_n$ such that $f(1,c_0, c_1, \dots, c_n)=C < \infty$. Let $D_C:=\{x \in {\mathbb R}^{n+2} : f(x) \leq C\}$ and $D'_C$ be the set of vectors in $D_C$ with first coordinate $1$. If $D_C$ is a bounded subset of ${\mathbb R}^{n+2}$, then so is $D'_C$. Because $D'_C$ is also non-empty and closed, and $f$ is lower-semicontinuous, this implies that the set of minimizers is non-empty.
	
	It is left to consider the case when the set $D_C$ is unbounded. Then there is a sequence $\{x_k\}_{k=0}^\infty$ such that $\lim\limits_{k\to \infty} \|x_k\|=\infty$ and $f(x_k)\leq C$ for all $k$, where $\|\cdot\|$ is the usual Euclidean norm in ${\mathbb R}^{n+2}$. Then $y_{k} = \frac{x_k}{\|x_k\|}$, $k=1,2,\dots$ are unit vectors belonging to the compact set $\{x \in {\mathbb R}^{n+2}: \|x\|=1\}$, hence, by passing to a subsequence if necessary, we may assume that $\lim\limits_{k\to \infty} y_k=y^*$ for some unit vector $y^* \in {\mathbb R}^{n+2}$. Now, for any $\lambda > 0$, let $K_\lambda$ be an integer such that $\|x_k\|\geq \lambda$ for all $k \geq K_\lambda$. Then the convexity of $f$ implies that
\begin{align*}
f(\lambda y_k) &= f\left(\left(1-\frac{\lambda}{\|x_k\|}\right)0+\frac{\lambda}{\|x_k\|}x_k\right)\\
&\leq \left(1-\frac{\lambda}{\|x_k\|}\right) f(0) + \frac{\lambda}{\|x_k\|} f(x_k)\\ &\leq \frac{\lambda}{\|x_k\|} C,    
\end{align*}
	for all $k \geq K_\lambda$. Hence,
	$$
		0 \leq \lim\limits_{k\to\infty} f(\lambda y_k) \leq \lim\limits_{k\to\infty} \frac{\lambda}{\|x_k\|} C = 0, 
	$$
	from which we conclude that $\lim\limits_{k\to\infty} f(\lambda y_k)=0$. Now lower-semicontinuity of $f$ implies that
	$$
		0 \leq f(\lambda y^*) = f\left(\lim\limits_{k\to \infty} (\lambda y_k)\right) \leq \lim\limits_{k\to \infty} f(\lambda y_k) = 0,
	$$
	hence $f(\lambda y^*)=0$ for all $\lambda > 0$. Let us write $y^*$ in the coordinate form as $y^*=(d^*, c_0^*, \dots, c_n^*)$. If $d^*=0$, then
	$$
		0 = f(\lambda y^*) = {\cal E}\left(\lambda X^*\right) \quad \text{for all} \quad \lambda>0, \quad \text{where} \quad  X^* = -c^*_0-\sum_{i=1}^n c^*_i X_i \in {\cal X},
	$$
	which is a contradiction with (E2) unless $X^*=0$. By the linear independence condition, $X^*=0$ is possible only if $c^*_0=\dots=c^*_n=0$, but this is a contradiction with $\|y^*\|=1$.
	
	Hence, $d^*\neq 0$. 
Let
	$
	Z^*:=d^*Y-c_0^*-\sum_{i=1}^n c_i^*X_i.
	$
	Then
	$$
	0=f(\lambda y^*)={\cal E}(\lambda Z^*) \quad \text{for all} \quad \lambda>0.
	$$
	By axiom (E2), this is possible only if $Z^*=0$. Therefore
	$$
	{\cal E}\left(Y-\frac{c_0^*}{d^*}-\sum_{i=1}^n \frac{c_i^*}{d^*}X_i\right)={\cal E}\left(Z^*/d^*\right) = {\cal E}(0) = 0,
	$$
    which implies that the minimum in \eqref{eq:linrerg} is $0$ and the set of minimizers is non-empty. 
\end{proof}

\begin{remark}
 It is easy to check that Proposition \ref{prop:linreg} does not hold without condition (E2). Indeed, fix any non-constant $X \in {\cal X}$ and consider functional
$$
{\cal E}(Z) = \begin{cases}
	0, & \text{if } Z = 0; \\
	\frac{a^2}{b}, & \text{if } Z = aY + bX \text{ for some constants } a \geq 0 \text{ and } b > 0; \\
	+\infty, & \text{otherwise.}
\end{cases}
$$  
The convexity of ${\cal E}(Z)$ follows from the convexity of function $f(a,b)=\frac{a^2}{b}$ in the region $\{a\geq 0, \, b>0\}$. Also, ${\cal E}$ is closed and satisfies (E1). On the other hand, ${\cal E}(Z)=0$ if and only $Z=bX$ for some $b>0$. If $Y \not \in {\cal X}$, then $Y-c_0-\sum_{i=1}^n c_i X_i$ is never of this form, hence the objective function in \eqref{eq:linrerg} is never $0$. However, for any $b>0$, ${\cal E}(Y+bX) = \frac{1}{b}$, hence, if $b\to \infty$, then the objective function in \eqref{eq:linrerg} can be arbitrarily close to $0$. Hence, the set of minimizers in this example is the empty set.    
\end{remark}

The regression decomposition theorem (Theorem 3.2 in \cite{MR2442649}) remains valid for any subregular error measure, with essentially the same proof.

Now we turn to conditional statistic tracking. For this, we need the following definition. 
\begin{definition}[Representation of risk identifiers, \cite{MR3357644}] A risk identifier $Q^Y$ at $Y \in \cL^p $ for a subregular measure of risk will be called \emph{representable} if there exists a Borel-measurable function $h^Y : \bbr \to \bbr$, possibly depending on $Y$, such that $$Q^Y(\omega) = h^Y(Y(\omega)) \text{for a.e.} \omega \in \Omega.$$
    
\end{definition}
The following is a reformulation of Theorem 5.1 of \cite{MR3357644}.
\begin{theorem}[Statistic tracking in regression]
For given $c_0^*\in \bbr$ and $\mathbf{c}^* \in \bbr^n,$ assume that 
\begin{equation}
    Y(\omega) = c_0^* + \mathbf{c}^{*\top}\mathbf{X} + \varepsilon(\omega) \quad \text{for all} \omega \in \Omega 
\end{equation}
with $\varepsilon \in \cL^p$ independent of $X_i, \ i = 1,\ldots,n$ and $\cS(\varepsilon) = 0.$ Let $(\cR, \cD, \cV, \cE)$ be a subregular quadrangle quartet. If $\cR$ has a representable risk identifier at $\varepsilon$ and $\varepsilon \in \operatorname{int}(\dom \cR),$ then 
\begin{equation}
    c_0^* +  \mathbf{c}^{*\top}\mathbf{X} \in \cS(Y|\mathbf{X}) \ \text{a.s.}
\end{equation}
\end{theorem}
The proof of the theorem remains the same.

\section{Stochastic Optimization: Models and Algorithms}\label{sec: SO and DR}
This section reviews and complements prior work on robust and distributionally robust optimization (DRO) from the RQ perspective. Leveraging the axiomatic theory of stochastic divergences developed in Section \ref{sec:gen families}, we show how DRO is naturally captured within RQ. We also outline recent advances in numerical optimization, focusing on epi-regularization. 

\subsection{Distributionally Robust Optimization}
The notion of risk has become central in modern stochastic optimization and closely related fields such as machine learning. Whenever the uncertainty is modeled probabilistically, the decision-maker aims to select a decision that minimizes the risk of future losses, i.e., solve
\begin{equation}\label{opt:primal risk}
    \min_{\mathbf{w}\in \mathcal{W}} \quad \cR(\ell(\mathbf{w},\omega)),
\end{equation}
where $\cR: \cL^p(\Omega, \cM, \bbp_0) \to (-\infty, \infty]$ is a subregular 
risk measure and $\ell: \mathcal{W} \times \Omega \to \bbr$ is a real-valued 
random loss function assumed to be convex in $w$ over a closed, convex, 
and nonempty set $\mathcal{W} \subseteq \bbr^n.$

Another popular idea related to the mitigation of the uncertainty in decision-making is DRO, i.e, 

\begin{equation}\label{eq:DRO}
   \min_{\mathbf{w}\in \mathcal{W}} \max_{\bbp \in \mathfrak{P}_\tau^\cG} \quad \bbe_{\bbp}[\ell(\mathbf{w},\omega)]
\end{equation}
where $\mathfrak{P}_\tau^\cG = {\left\{\bbp \ll \bbp_0: \cG(\bbp  \parallel \bbp_0) \leq \tau\right\}}$ with $\cG$ being a ``distance'' between $\bbp$ and $\bbp_0$. By definition, this distance is related to the stochastic divergence $\cJ$ as follows
$$\cG(\bbp \parallel \bbp_0):= \cJ(Q) \, , $$
where $Q = d \bbp/ d \bbp_0$ is the Radon--Nikodym density. 

Formulation \eqref{eq:DRO} admits an equivalent representation in terms of a subregular \emph{parent} risk measure $\cR$ associated with stochastic divergence $\cJ$. Indeed, given a stochastic divergence $\cJ$, its conjugate $\cR=\cJ^*$ is given by 

\begin{equation*}
    \cR(\ell(\mathbf{w}, \omega)): = \max_{Q \in \cQ} \quad \left(\bbe_{\bbp_0}[Q\ell(\mathbf{w},\omega)] - \cJ(Q)\right) \, , 
\end{equation*}
with $\cQ := \left\{Q \in \cL^q: Q \geq 0, \  \bbe_{\bbp_0}[Q] = 1\right\}$. 
Because $\cJ$ is a stochastic divergence, its corresponding family $\cQ_\tau = \cQ_\tau^\cJ$ defined in {\eqref{eq:qtaudef}} can be written as
$$\cQ_\tau^\cJ: = \left\{Q \in \cQ: \cJ(Q) \leq \tau  \right\}$$
and notice that there is a one-to-one correspondence between $\cQ_\tau^\cJ$ and $\cP_\tau^\cG$. Therefore, problem \eqref{eq:DRO} is equivalent to minimizing the support functional $\cR_\tau$ of the set $\cQ_\tau^\cJ$, i.e., 
\begin{equation}\label{eq: DRO level set}
 \min_{\mathbf{w} \in \mathcal{W}} \quad \cR_\tau(\ell(\mathbf{w}, \omega)) \, 
\end{equation}
where

$$
\cR_\tau(\ell(\mathbf{w},\omega))=\sup_{Q \in \cQ_\tau^\cJ} {\mathbb E}[Q \ell(\mathbf{w},\omega)] = \inf\limits_{\lambda>0} \lambda[\cR(\lambda^{-1}\ell(\mathbf{w},\omega))+\tau],
$$
see {\eqref{eq:dualtau}} and {\eqref{eq:basic}}.
Moreover, functional $\cR_\tau$ is itself positively homogeneous and monotone subregular risk (see Proposition \ref{prop:families}).

In modern distributionally robust optimization \cite{shapiro2017dro}, the decision-maker first selects $\cJ$ and then solves  \eqref{eq:DRO}. This is equivalent to choosing the parent risk $\cR$ and solving \eqref{eq: DRO level set}. Moreover, the option of choosing a known risk first may be beneficial, as there might be an efficient way to optimize it through regret. Proposition \ref{prop:projfam} implies that \eqref{eq: DRO level set} can be rewritten as 
\begin{equation}\label{opt:level set regret}
   \min_{\mathbf{w}\in \mathcal{W}, \ C \in \bbr} \quad \left(C + \cV_\tau(\ell(\mathbf{w},\omega) - C)\right). 
\end{equation}
Indeed, formulation \eqref{opt:level set regret} is usually more computationally efficient than \eqref{eq:DRO}. However, formulation \eqref{eq:DRO} turned out to be crucial in the context of epi-regularization and numerical algorithms associated with it, which is covered in the next subsection.

\subsection{Epi-Regularization and Applications}

Subregular (or regular, or coherent) risk functionals central to decision-making are, by definition, convex yet frequently nonsmooth, which complicates analysis and large-scale computation. Epi-smoothing via infimal convolution with sufficiently smooth \emph{kernels} \cite{Burke2017} --- also called epi-regularization \cite{kouri2020epi} --- yields better-conditioned surrogates that preserve \emph{core} properties and support faster-converging algorithms. This perspective has been leveraged in risk-averse modeling and algorithms by Kouri and Surowiec \cite{kouri2020epi}, including PDE-constrained applications, providing a principled pathway from theory to numerics.

\begin{definition}[Epi-regularized subregular risk]\label{def:primal epi-reg def}
   Let $\mathcal{R}:\cL^p \to (-\infty, \infty]$ be a subregular measure of risk, and let $\tilde{\cV}:\cL^p \to (-\infty, \infty]$ be a proper, closed, and convex 
   functional.
   For $\epsilon > 0$, the epi-regularized risk measure is defined as
   \begin{equation}\label{eq:epiriskinfimal}
     \mathcal{R}^{\tilde{\cV}}_\epsilon(X) := \inf_{Y \in \cL^p} \left(\mathcal{R}(X-Y) + \frac{1}{\epsilon}\tilde{\cV}(\epsilon Y) \right)= \inf_{Y \in \cL^p}  \left(\mathcal{R}(Y) + \frac{1}{\epsilon}\tilde{\cV}(\epsilon (X-Y))\right)\, .  
   \end{equation}
\end{definition}

Given Definition~\ref{def:primal epi-reg def}, it is natural to ask under what axioms on $\tilde{\cV}$ the functional $\mathcal{R}^{\tilde{\cV}}_\epsilon$ defines a subregular risk. Kouri and Surowiec \cite{kouri2020epi} partially addressed this question for regular risk measures. We extend their results to subregular risk measures within an axiomatic framework.

\begin{proposition}\label{prop:epirisk} Let $\mathcal{R}:\cL^p \to (-\infty, \infty]$ be a subregular measure of risk, and let $\tilde{\cV}:\cL^p \to (-\infty, \infty]$ be a subregular regret such that $\dom \cR^* \subseteq \dom \tilde{\cV}^*$. Then the epi-regularized risk  $\cR^{\tilde{\cV}}_\epsilon$ is subregular. 
\end{proposition}
\begin{proof}
    Let us show that $\cR^{\tilde{\cV}}_\epsilon$ is subregular by verifying the axioms of Definition~\ref{def:risk}. First, we observe that since $\cR^{\tilde{\cV}}_\epsilon$ is the infimal convolution of two closed and convex functionals $\cR$ and $\tilde{\cV}$ such that $\dom \cR^* \subseteq \dom \tilde{\cV}^*$, then it is also closed and convex \cite[Theorem 9.4.2]{AttouchButtazzoMichaille2006}. Thus, it is left to show (R1) and (R2) from Definition~\ref{def:risk}.

    \paragraph{(R1)} Assumption $\dom \cR^* \subseteq \dom \tilde{\cV}^*$ and Fenchel--Moreau theorem guarantees that (see \cite[Proposition 9.3.5]{AttouchButtazzoMichaille2006}) $ \cR^{\tilde{\cV}}_\epsilon$
is equal to its biconjugate, which, in turn, satisfies
    \begin{equation}\label{eq:conjepirisk}
       \cR^{\tilde{\cV}}_\epsilon(X) = (\cR^{\tilde{\cV}}_\epsilon)^{**}(X) = \sup_{Q \in \dom \cR^*}  \left(\ \bbe[QX] - \cR^*(Q) - \epsilon^{-1}\tilde{\cV}^*(Q) \right)\, . 
    \end{equation}
    Hence since $\cR$ is a subregular risk and $\tilde{\cV}$ is a subregular regret, then  Proposition~\ref{prop:conjugate} implies that $1 \in \dom \cR^*$, and 
    \begin{itemize}
        \item[(a)] 
        \begin{align*}
          \cR^{\tilde{\cV}}_\epsilon(C) &= \sup\limits_{Q \in \dom \cR^*}  \left(\ \bbe[QC] - \cR^*(Q) - \epsilon^{-1}\tilde{\cV}^*(Q) \right)  \\
          &= C - \inf\limits_{Q \in \dom \cR^*}  \left(\ \cR^*(Q) + \epsilon^{-1}\tilde{\cV}^*(Q) \right)\\
          &=C \, .
        \end{align*}

        \item[(b)]
        \begin{align*}
          \cR^{\tilde{\cV}}_\epsilon(X) &= \sup\limits_{Q \in \dom \cR^*}  \left(\ \bbe[QX] - \cR^*(Q) - \epsilon^{-1}\tilde{\cV}^*(Q) \right)  \\  
          &\geq  \bbe[QX] - \cR^*(Q) - \epsilon^{-1}\tilde{\cV}^*(Q), \quad \forall \; Q \in \dom \cR^*\\
          & = \bbe[X], \text{ for } Q = 1 \, .
        \end{align*}
    \end{itemize}
       
\paragraph{(R2)} Fix 
any non-constant $X\in \cL^p$. Then since $\cR$ is subregular, Proposition~\ref{prop:conjugate} implies that $\E[X\bar Q]>\E[X]$  
for some $\bar Q \in \dom \cR^*$. Further, \eqref{eq:conjepirisk} implies that for any scalar 
$\lambda > \frac{\cR^*(\bar Q) + \epsilon^{-1}\tilde{\cV}^*(\bar Q)}{\E[X\bar Q]-\E[X]}$,
	$$
		\cR^{\tilde{\cV}}_\epsilon(\lambda X) = \sup_{Q \in 
  \dom \cR^*}\left(\lambda\E[XQ]-\cR^*(Q) - \epsilon^{-1}\tilde{\cV}^*( Q)\right) 
  $$
  $$
          \geq \lambda\E[X\bar Q]-\cR^*(\bar Q) - \epsilon^{-1}\tilde{\cV}^*(\bar Q) > \E[\lambda X] \, ,
	$$
    which completes the proof.
\end{proof}
For the differentiability properties of $\cR^{\tilde{\cV}}_\epsilon$ see \cite[Theorem 2]{kouri2020epi}.
Obviously, since $\cR^{\tilde{\cV}}_\epsilon$ is subregular, there exists a quadrangle quartet $(\cR^{\tilde{\cV}}_\epsilon, \cV^{\tilde{\cV}}_\epsilon, \cD^{\tilde{\cV}}_\epsilon, \cE^{\tilde{\cV}}_\epsilon)$ containing it. Although such a quartet is not unique, there is a natural way of constructing it given a quadrangle corresponding to $\cR$. 

\begin{proposition}\label{prop:epiquadrangle}
Let $(\cR, \cV, \cD, \cE)$ be a subregular quadrangle, and let $\tilde{\cV}:\cL^p \to (-\infty, \infty]$ be a subregular regret such that $\dom \cR^* \subseteq \dom \tilde{\cV}^*$. Then $\cV^{\tilde{\cV}}_\epsilon$ given by 
\begin{equation}\label{eq:epiregretinfimal}
     \mathcal{V}^{\tilde{\cV}}_\epsilon(X) := \inf_{Y \in \cL^p} \left(\mathcal{V}(X-Y) + \frac{1}{\epsilon}\tilde{\cV}(\epsilon Y) \right) = \inf_{Y \in \cL^p}  \left(\mathcal{V}(Y) + \frac{1}{\epsilon}\tilde{\cV}(\epsilon (X-Y))\right)\, ,  
   \end{equation}
   generates (projects) a subregular quadrangle $(\cR^{\tilde{\cV}}_\epsilon, \cV^{\tilde{\cV}}_\epsilon, \cD^{\tilde{\cV}}_\epsilon, \cE^{\tilde{\cV}}_\epsilon)$ with $\cR^{\tilde{\cV}}_\epsilon$ given by \eqref{eq:epiriskinfimal}.
\end{proposition}
\begin{proof}
By Definition \ref{def:primal epi-reg def} and axiom (Q2) in Definition \ref{def:subreg quadrangle},    
\begin{align*}
     \mathcal{R}^{\tilde{\cV}}_\epsilon(X) &= \inf_{Y \in \cL^p} \left( \mathcal{R}(X-Y) + \frac{1}{\epsilon}\tilde{\cV}(\epsilon Y) \right)\\
     &= \inf_{Y \in \cL^p} \min_{C}  \left( C + \mathcal{V}(X-Y -C) + \frac{1}{\epsilon}\tilde{\cV}(\epsilon Y) \right)\\
     &= \min_{C}  \left( C +  \inf_{Y \in \cL^p} \mathcal{V}(X-C-Y) + \frac{1}{\epsilon}\tilde{\cV}(\epsilon Y) \right)\\
     &= \min_{C} \left( C + \mathcal{V}^{\tilde{\cV}}_\epsilon(X-C) \, \right) . 
\end{align*} 
Setting $\cD^{\tilde{\cV}}_\epsilon(X) = \cR^{\tilde{\cV}}_\epsilon(X) -\bbe[X]$ and $\cE^{\tilde{\cV}}_\epsilon(X) = \cV^{\tilde{\cV}}_\epsilon(X) - \bbe[X]$ completes the proof.
\end{proof}
\begin{remark}[Epi-regularized quadrangle]
Given a quadrangle $(\cR, \cV, \cD, \cE)$, with nonsmooth $\cR$, Proposition \ref{prop:epiquadrangle} implies that epi-regularizing the regret component $\cV$ suffices to epi-regularize the entire quadrangle.     
\end{remark}

Now, let us consider coherent quadrangles $(\cR, \cV, \cD, \cE)$. It is convenient to epi-regularize such quadrangles by utilizing the dual representation of their regret components.

\begin{proposition}\label{prop:cohreg}
    Let $\mathcal{V}:\cL^p \to (-\infty, \infty]$ be a coherent measure of regret, and let $\tilde{\cK}:\cL^q \to [0, \infty]$ be a divergence root such that $\dom \cV^* \subseteq \cl \dom \tilde{\cK}$. For any $\epsilon > 0$, functional 
    \begin{equation}\label{eq:epiregretdual}
       \cV^{\tilde{\cK}}_\epsilon(X) = \sup_{Q \in \dom \cV^*}  \left( \bbe[QX] - \epsilon^{-1}\tilde{\cK}(Q) \, \right).
   \end{equation}  
   is a coherent regret measure in the general sense (cf. Definition \ref{def:coherent_regret}).
\end{proposition}
\begin{proof}
First, note that convexity and lower semicontinuity of $\cV^{\tilde{\cK}}_\epsilon$ follow directly from the definition \eqref{eq:epiregretdual}.

Next, we verify that $\cV^{\tilde{\cK}}_\epsilon(0)=0$. Indeed,
\begin{align*}
  \cV^{\tilde{\cK}}_\epsilon(0)
  &= \sup_{Q \in \dom \cV^*} \big( \bbe[Q \cdot 0] - \epsilon^{-1}\tilde{\cK}(Q) \big)\\
  &= \sup_{Q \in \dom \cV^*} \big( - \epsilon^{-1}\tilde{\cK}(Q) \big)\\
  &= - \inf_{Q \in \dom \cV^*} \big( \epsilon^{-1}\tilde{\cK}(Q) \big)
  = 0,
\end{align*}
where the last equality uses $1 \in \dom \cV^*$ and a divergence root property  $\tilde{\cK}(1)=0$.
Monotonicity of $\cV^{\tilde{\cK}}_\epsilon$ follows from its definition together with the fact that 
$$\dom \cV^* \subseteq \cl \dom \tilde{\cK} = \{Q \in \cL^q : Q \ge 0\} \, .
$$
\end{proof}
\begin{definition}[Epi-regularized coherent regret]\label{def:dual epi-reg def}
Under the assumptions of Proposition~\ref{prop:cohreg}, the functional~\eqref{eq:epiregretdual} is referred to as the \emph{epi-regularized regret measure}. 
\end{definition}

Epi-regularized regret $\cV^{\tilde{\cK}}_\epsilon$ naturally induces a quadrangle $(\cR^{\tilde{\cK}}_\epsilon, \cV^{\tilde{\cK}}_\epsilon, \cD^{\tilde{\cK}}_\epsilon, \cE^{\tilde{\cK}}_\epsilon)$, where $\cR^{\tilde{\cK}}_\epsilon(X) = \min\limits_C\left(C+ \cV^{\tilde{\cK}}_\epsilon(X-C)\right)$ is a coherent risk measure in the general sense with error and deviation defined by subtracting the expectation from regret and risk, respectively. 
\newpage
\bibliographystyle{plain}
\bibliography{references.bib}
\end{document}